\documentclass[12pt]{article}
\usepackage{amsmath,amsthm,amsfonts,amssymb, xcolor}
\RequirePackage[colorlinks,citecolor=blue,urlcolor=blue]{hyperref}
\usepackage[shortlabels]{enumitem}

\theoremstyle{plain}
\newtheorem{theorem}{Theorem}[section]

\newtheorem{lemma}[theorem]{Lemma}
\newtheorem{proposition}[theorem]{Proposition}

\theoremstyle{definition}
\newtheorem{definition}[theorem]{Definition}

\newtheorem{remark}[theorem]{Remark}

\newtheorem{example}[theorem]{Example}

\newtheorem*{assumptions*}{\assumptionnumber}
\providecommand{\assumptionnumber}{}
\makeatletter
\newenvironment{assumptions}[1]
 {%
	 \renewcommand{\assumptionnumber}{Assumption #1}%
\begin{assumptions*}%
  \protected@edef\@currentlabel{#1}%
 }
 {%
  \end{assumptions*}
 }
\makeatother
\def\cD{\mathcal{D}}

\def\cF{\mathcal{F}}

\def\cH{\mathcal{H}}

\def\cL{\mathcal{L}}
\def\cM{\mathcal{M}}

\def\cS{\mathcal{S}}

\def\bC{\mathbb{C}}

\def\bE{\mathbb{E}}
\def\bN{\mathbb{N}}
\def\bP{\mathbb{P}}
\def\bR{\mathbb{R}}
\def\R{\mathbb{R}}

\def\e{\varepsilon}

\newcommand{\Norm}[1]{\left|\left|  #1   \right|\right|}
\newcommand{\InPrd}[1]{\left\langle #1 \right\rangle}
\def\E{{\mathbb E}}

\def\bs{{\bf s}}

\def\bx{{\bf x}}
\def\by{{\bf y}}

\def\e{{\rm e}}

\def\cprime{$'$}

\numberwithin{equation}{section}
\begin{document}

\title{
  Exact asymptotics of the stochastic wave equation \\
  with time-independent noise
}

\author{Raluca M. Balan\footnote{Corresponding author. Department of Mathematics and Statistics, University of Ottawa,  Ottawa, ON K1N 6N5, Canada. E-mail address:
rbalan@uottawa.ca}
	\footnote{Research supported by a grant from Natural Sciences and Engineering Research Council of Canada}
	\and
	Le Chen\footnote{Department of Mathematics, Emory University, Atlanta, GA 30322. E-mail address: le.chen@emory.edu}
	\and
	Xia Chen\footnote{Department of Mathematics, University of Tennessee, Knoxville, TN 37996-1300. E-mail address: xchen3@math.utk.edu}
	\footnote{Research is partially supported by the Simons Foundation \#585506}}

\date{\today}
\maketitle

\begin{abstract}
	\noindent In this article, we study the stochastic wave equation in
	all dimensions $d\leq 3$, driven by a Gaussian noise $\dot{W}$ which does not
	depend on time. We assume that either the noise is white, or the
	covariance function of the noise satisfies a scaling property similar to
	the Riesz kernel. The solution is interpreted in the Skorohod sense
	using Malliavin calculus.  We obtain the exact asymptotic behaviour of
	the $p$-th moment of the solution either when the time is large or when
	$p$ is large.  For the critical case, that is the case when $d=3$ and the noise is
	white, we obtain the exact transition time for the second moment to be
	finite.
\end{abstract}


\vspace{1mm}


\tableofcontents

\section{Introduction}

In this paper, we study the following stochastic wave equation:
\begin{equation}
	\label{E:SWE}
	\left\{\begin{array}{rcl}
		\displaystyle \frac{\partial^2 u}{\partial t^2}(t,x) & = & \Delta u(t,x) + \sqrt{\theta}\,u(t,x)\dot{W}(x), \quad t>0,\: x \in \bR^d, \\[2ex]
		\displaystyle u(0,x) & = & 1, \qquad \displaystyle \frac{\partial u}{\partial t}(0,x)  =  0,
	\end{array}\right.
\end{equation}
where $\theta>0$ and $W=\{W(\varphi);\varphi \in \cD(\bR^d)\}$ is a
(time-independent) Gaussian process, defined on a complete probability space
$(\Omega, {\cal F},P)$, with mean zero and covariance:
\begin{align*}
	\E[W(\varphi) W(\psi)]
	=  \int_{\bR^d} \cF \varphi(\xi) \overline{\cF \psi(\xi)}\mu(d\xi)
	=: \langle \varphi,\psi \rangle_{\cH},
\end{align*}
with the {\it spectral measure} $\mu$ being assumed to be a non-negative and nonnegative definite tempered measure
\footnote{A Schwarz distribution $\mu \in \cS'(\bR^d)$ is \textit{nonnegative definite} if $(\mu,\phi*\phi^*)\ge 0$ for every $\phi\in\cS(\R^d)$,
where $*$ denotes the spatial convolution and $\phi^*(x)=\overline{\phi(-x)}$; see \cite[Section 3.1]{GV61-IV}.} on $\bR^d$.
Here, $\cF \varphi(\xi)=\int_{\bR^d}\e^{-i \xi \cdot x}\varphi(x)dx$ denotes the Fourier transform.
It is known that the Fourier transform of $\mu$, denoted by $\gamma$, is also a nonnegative and nonnegative definite tempered measure.
The solution is interpreted in the Skorohod sense, as explained in Section \ref{S:Pre} below.
We will focus on the cases when $d\le 3$ and denote by $G$ the corresponding fundamental solution:
\begin{align}
	G(t,x)=
	\begin{cases}
		\displaystyle \frac{1}{2}1_{\{|x|<t\}}                               & \text{if $d=1$},\\[1em]
		\displaystyle \frac{1}{2\pi} \frac{1}{\sqrt{t^2-|x|^2}}1_{\{|x|<t\}} & \text{if $d=2$},\\[1em]
		\displaystyle \frac{1}{4\pi t}\sigma_t                               & \text{if $d=3$},
	\end{cases}
\end{align}
where $\sigma_t$ is the surface measure on the sphere
$\{x\in \bR^3; |x|=t\}$ with $|\cdot|$ being the Euclidean norm in $\bR^d$.

\medskip

The goal of this article is to derive sharp solvability conditions and the
precise moment asymptotics of the Skorohod solution to equation \eqref{E:SWE}.  The
literature for the stochastic heat equation with Gaussian noise is immense.
Here we point out \cite{xchen17}, \cite{xchen19}, \cite{CHSX15} and \cite{HLN}
for an incomplete list of references.  But to the best of our knowledge, the
investigations related to the wave equation under the current setting seem to be
new in the literature.  Compared to the stochastic heat equation, the main
difficulty is the absence of the Feynman-Kac representation of the moments.  The
singularity of the wave kernel especially when $d=3$ also causes some technical
difficulties.  The method proposed in this paper turns out to be fairly general,
which can be used to study many other stochastic partial differential equations
provided that the corresponding fundamental solution is nonnegative and has a
certain scaling property.

\bigskip

There are many references dedicated to the study of the stochastic wave equation
with time-dependent spatially-homogeneous Gaussian noise and possibly a
Lipschitz non-linear term $\sigma(u(t,x))$ multiplying the noise.  These
investigations started with Dalang's seminal article \cite{dalang99} for the
case $d\leq 3$, and were extended to the case $d\geq 4$ in Conus and Dalang
\cite{conus-dalang09}. In the case when $\sigma(u)=u$, Mueller and Dalang
\cite{dalang-mueller09} obtained exponential bounds for the $p$-th moments of
the solution in dimension $d=3$. An exact formula for the second moment of the
solution of the stochastic wave equation with space-time white noise was
obtained in Chen and Dalang \cite{CD15}, from which one easily obtains the large
time asymptotics for the second moment.  All these references handle the
time-white noise. For time-colored noise, one may refer to the recent works by
Balan and Song \cite{balan-song19,balan-song17}.  The current work serves as the
first necessary step to understand the exact asymptotic property of the
stochastic wave equation with more general, i.e., time-dependent, noises.
\medskip

Typical examples of the noise $\dot{W}$ include the white noise, the Riesz
kernel noise, fractional noise, a hybrid of these noises, etc.  In order to
cover all these examples, we will work under the following three assumptions ---
Assumptions \ref{A:Main}, \ref{A:White}, \ref{A:Critical} ---  exclusively:

\begin{assumptions}{A}
	\label{A:Main}
	Assume that:
	\begin{enumerate}[(i)]
		\item both $\mu$ and $\gamma$ are absolutely continuous with respect to the Lebesgue measure;
		\item for some $\alpha \in (0,d)$, $\gamma$ satisfies the following scaling property:
			\begin{equation}
				\label{E:scaleGamma}
				\gamma(cx)=c^{-\alpha}\gamma(x) \qquad \text{for all $c>0$, $x\in \bR^d$};
			\end{equation}
		\item there exists a nonnegative function $K$ on $\R^d$ such that
			\begin{align}
				\label{E:gamma=K*K}
				\gamma = K*K,
			\end{align}
			where ``$*$" refers to the spatial convolution.
	\end{enumerate}
\end{assumptions}

\begin{remark}
  Part (i) of Assumption \ref{A:Main} implies that there exists a nonnegative
  definite (possibly signed) function $K$ such that the decomposition
  \eqref{E:gamma=K*K} holds.  Indeed, let $\varphi(\xi)$ be the density of
  $\mu$.  It is nonnegative and hence $\sqrt{\varphi(\xi)}$ is a well defined
  function.  Viewed as a nonnegative tempered measure, its inverse Fourier
  transform, denoted as $K:=\cF^{-1}\left(\sqrt{\varphi}\right)$, is a
  nonnegative definite (tempered) measure. Therefore, we have
  \eqref{E:gamma=K*K} in the sense of distributions.  Moreover, the absolute
  continuity of $\gamma$ entails that $K$ also admits a density.  However,
  the existence of a nonnegative $K$ is not immediate. Hence, we list it as part
  (iii) of the assumption.  Part (iii) is always satisfied for all examples that
  we are interested.
\end{remark}

Assumption~\ref{A:Main} excludes the white noise case, which will be treated
separately in this paper in two cases: the sub-critical case -- Assumption
\ref{A:White} and the critical case -- Assumption \ref{A:Critical}.  Formally,
both Assumptions~\ref{A:White} and \ref{A:Critical} below correspond to
Assumption \ref{A:Main} with $\mu(d\xi)=(2\pi)^{-d}d\xi$, $\alpha=d$, and
$\gamma=K=\delta_0$.
\begin{assumptions}{B}
	\label{A:White}
	Assume that $d\le 2$, $\gamma=\delta_0$ (or equivalently $\mu(d\xi)=(2\pi)^{-d}d\xi$).
\end{assumptions}
\begin{assumptions}{C}
	\label{A:Critical}
	Assume that $d=3$, $\gamma=\delta_0$ (or equivalently $\mu(d\xi)=(2\pi)^{-3}d\xi$).
\end{assumptions}

\begin{remark}
  Under Assumption \ref{A:Main}, the scaling property of $\gamma$ implies that
  $\mu(d\xi)=\varphi(\xi)d\xi$ has the following scaling property:
	\begin{align}
		\label{E:scaleMu}
		\varphi(c\xi) = c^{-\left(d-\alpha\right)}\varphi(\xi) \quad \text{or equivalently} \quad
		\mu(cA)=c^{\alpha}\mu(A)
	\end{align}
  for all $c>0$, $\xi \in \bR^d$ and $A\in \mathcal{B}(\R^d)$.  In the case of
  the white noise, i.e., under either Assumption \ref{A:White} or
  \ref{A:Critical}, the measure $\mu$ has the scaling property:
	\begin{align*}
		\mu(cA)=c^{d}\mu(A) \quad \mbox{for all} \ c>0, A \in {\cal B}(\bR^d).
	\end{align*}
\end{remark}

Assumption~\ref{A:Main} is satisfied by the following examples:
\begin{example}
	\label{E:Riesz}
	$\gamma(x)=|x|^{-\alpha}$  for some $\alpha \in (0,d)$.
	In this case, $\mu(d\xi)=C_{d,\alpha}|\xi|^{-(d-\alpha)}d\xi$ and $K(x)=\beta_{\alpha,d}|x|^{-(d+\alpha)/2}$, where
	\begin{equation}
		\label{E:constCBeta}
		C_{d,\alpha}=\pi^{-d/2} 2^{-\alpha} \frac{\Gamma((d-\alpha)/2)}{\Gamma(\alpha/2)} \quad \mbox{and} \quad
		\beta_{\alpha,d}=\pi^{-d/4} \frac{\Gamma((d+\alpha)/4)}{\Gamma((d-\alpha)/4)} \sqrt{\frac{\Gamma((d-\alpha)/2)}{\Gamma(\alpha/2)}}.
	\end{equation}
	The heat equation with this noise was studied in Hu et al \cite{HHNT}.
\end{example}

\begin{example}
	\label{E:fBM}
  $\gamma(x)=\prod_{i=1}^{d} |x_i|^{-\alpha_i}$ for some
  $\alpha_1,\ldots,\alpha_d \in (0,1)$.  In this case,
  $\mu(d\xi)=\prod_{i=1}^{d}(C_{1,\alpha_i}|\xi_i|^{-(1-\alpha_i)}d\xi_i)$ and
  $K(x)=\prod_{i=1}^{d}(\beta_{\alpha_i,1} |x_i|^{-(1+\alpha_i)/2})$, where
  $C_{1,\alpha_i}$ and $\beta_{\alpha_i,1}$ are given by \eqref{E:constCBeta}.
  The function $\gamma$ satisfies the scaling relation \eqref{E:scaleGamma} with
  $\alpha=\sum_{i=1}^{d}\alpha_i$.  In the parametrization $\alpha_i=2-2H_i$
  with $H_i \in (1/2,1)$, the noise $W$ corresponds to the fractional Brownian
  sheet with indices $H_1,\ldots,H_d$.  The heat equation with this noise was
  studied in Hu \cite{hu01}.
\end{example}


\begin{example}
	\label{E:Hybrid}
  In the case $d=3$, one can construct another example by grouping coordinates and combining the previous two examples, i.e. letting
    $\gamma(x)=|x_1|^{-\alpha_1}|(x_2,x_3)|^{-\alpha_2}$ for $x=(x_1,x_2,x_3)\in \R^3$, with $\alpha_1 \in (0,1)$ and $\alpha_2 \in (0,2)$. Then
    $\mu(d\xi)=C_{1,\alpha_1}C_{2,\alpha_2}|\xi_1|^{-(1-\alpha_1)}\linebreak
    |(\xi_2,\xi_3)|^{-(2-\alpha_2)}$. The scaling property
    \eqref{E:scaleGamma} is satisfied with $\alpha=\alpha_1+\alpha_2\in(0,3)$.
\end{example}

\bigskip

We introduce the following variational quantity:
for any nonnegative definite measure $f$ and $\theta>0$, define
\begin{equation}
	\label{E:cM}
	\cM(f,\theta)=\sup_{g \in \cF_d} \left\{
		\InPrd{g^2*f,g^2}^{1/2}
		-\frac{\theta}{2} \int_{\bR^d}|\nabla g(x)|^2 dx
	\right\},
\end{equation}
where $*$ is the convolution, $\langle \cdot,\cdot \rangle$ is the inner product in $L^2(\bR^d)$,
\begin{align*}
	\cF_d=\left\{g \in W^{1,2}(\bR^d);\Norm{g}_{L^2(\R^d)}=1\right\}
\end{align*}
and $W^{1,2}(\bR^d)$ is the Sobolev space. This quantity is related to the large deviation results for the local time of stable processes. See for instance \cite{chen-li-rosen05} and the references therein.

We denote $\cM=\cM(\gamma,1)$ and $\cM(f) = \cM(f,1)$.
In particular, for $d=1$, we have ${\cal M}(\delta_0)=(3/4)(1/6)^{1/3}$, which follows by Lemma 7.2 of \cite{CL04} with $p=2$.

\medskip

Throughout the article, for any $p>0$, we use $\Norm{\cdot}_p$ to denote the $L^p(\Omega)$-norm.

Here are the main results of this article.

\begin{theorem}
	\label{T:Main}
  If either Assumption \ref{A:Main} holds with $0<\alpha<d \leq 3$ or Assumption
  \ref{A:White} holds, then equation \eqref{E:SWE} has a unique solution
  $u(t,x)$ in $L^p(\Omega)$ for all $p\ge 2$, $t>0$ and $x\in\R^d$.  Moreover,
  we have the following moment asymptotics:
	\begin{enumerate}[(i)]
		\item Under Assumption \ref{A:Main} with $0<\alpha<d\le 3$, by setting
			\begin{align}
				\label{E:tp}
				t_p := (p-1)^{1/(4-\alpha)}t,
			\end{align}
			it holds that
			\begin{align}
				\label{E:Main}
				\lim_{t_p\to \infty}t_p^{-\frac{4-\alpha}{3-\alpha}}\log \Norm{u(t,x)}_p
				&= \theta^{\frac{1}{3-\alpha}}
			    	   \left(\frac{1}{2}\right)^{\frac{\alpha}{2(3-\alpha)}}
				   \frac{3-\alpha}{2}
				   \left(\frac{2\cM^{1/2}}{4-\alpha}\right)^{\frac{4-\alpha}{3-\alpha}}.
			\end{align}
			In particular, by freezing an arbitrary $p\ge 2$ in \eqref{E:Main},
			\begin{equation}
				\label{E:tMain}
				\lim_{t\to \infty} t^{-\frac{4-\alpha}{3-\alpha}} \log \E \left(|u(t,x)|^p\right)
				=p (p-1)^{\frac{1}{3-\alpha}} \theta^{\frac{1}{3-\alpha}}
				\left(\frac{1}{2} \right)^{\frac{\alpha}{2(3-\alpha)}} \frac{3-\alpha}{2}
				\left( \frac{2\cM^{1/2}}{4-\alpha}\right)^{\frac{4-\alpha}{3-\alpha}},
			\end{equation}
			and by freezing an arbitrary $t>0$ in \eqref{E:Main},
			\begin{equation}
				\label{E:pMain}
				\lim_{p\to \infty} p^{-\frac{4-\alpha}{3-\alpha}} \log \E\left(|u(t,x)|^p\right)
				= t^{\frac{4-\alpha}{3-\alpha}}\theta^{\frac{1}{3-\alpha}}
				\left(\frac{1}{2} \right)^{\frac{\alpha}{2(3-\alpha)}}\frac{3-\alpha}{2}
				\left( \frac{2\cM^{1/2}}{4-\alpha}\right)^{\frac{4-\alpha}{3-\alpha}}.
			\end{equation}
		\item Under Assumption \ref{A:White}, \eqref{E:Main} -- \eqref{E:pMain} are still true
			provided that all $\alpha$'s and $\cM$'s in part (i) are replaced by $d$ and $\cM(\delta_0)$, respectively.
	\end{enumerate}
\end{theorem}

The case when $\dot{W}(x)$ is white (i.e., $\gamma(\cdot)=\delta_0(\cdot)$) is of particular
interest due to its root in physics and also due to the recent investigation of the {\it
Kardar-Parisi-Zhang (KPZ) equation} \cite{KardarParisiZhang1986}.  In the setting of parabolic
equations, only the one-dimensional equation run by white noise has a full-scale solvability. The
two and higher dimensions have to be treated with great care. The large-time asymptotic behaviour of
the moments of the Skorohod solution of the parabolic Anderson model with white noise in time has
been stipulated as conjecture (1.19) in \cite{xchen17} (in arbitrary dimension $d$). As far as we
know, this conjecture remains an open problem.

 Here we also like to mention the
paper by  Hairer and Labb\'e \cite{HL} on the solvability for a two-dimensional parabolic Anderson
equation with time-independent and space-white noise, where the pathwise solution is
constructed after some logarithmic renormalization.  In the setting of hyperbolic equations, we
have seen from Theorem \ref{T:Main} that the system behaves ``normally'' in the case of
two-dimensional white noise, namely, there exists an $L^2(\Omega)$ solution for all
$t>0$ and $x\in\bR^2$.  Moreover, Theorem \ref{T:Critical} below shows that in case of
three-dimensional white noise, we only have short-time existence of $L^2(\Omega)$
solutions.

In order to state the next theorem, we first need to introduce some notation.
Assume that $d=\alpha=3$. For any $\theta>0$ and $p\ge 2$, we define the {\it
critical time for the $p$-th moment} as follows:
\begin{align}
	\label{E:Tp}
	T_p=T_p(\theta) := \frac{\sqrt{2}}{\theta (p-1)  \cM(\delta_0)^{1/2}}.
\end{align}
Recall that the solution to \eqref{E:SWE} is interpreted in terms of the Wiener
chaos expansion and if the solution $u(t,x)$ exists in $L^2(\Omega)$, its second
moment has to be equal to
\begin{align}
	\label{E:SecondM}
	\E\left(|u(t,x)|^2\right)
	=\sum_{n\geq 0} \theta^n n! \Norm{\widetilde{f}_n(\cdot,x;t)}_{\cH^{\otimes n}}^2<\infty;
\end{align}
see Section \ref{S:Pre} for the notation.
\begin{theorem}
	\label{T:Critical}
	Under Assumption \ref{A:Critical}, we have the following two scenarios:
	\begin{enumerate}[(i)]
    \item If $t<T_2$, then for all $x\in\R^d$, the series in \eqref{E:SecondM}
      converges.  Hence, there exists a unique solution $\left\{u(t,x):\: t\in
      \left(0,T_2\right),\: x\in\R^d\right\}$ to \eqref{E:SWE} in $L^2(\Omega)$.
      In general, for any $p\ge 2$, there exists a unique solution
      $\left\{u(t,x):\: t\in \left(0,T_p\right),\: x\in\R^d\right\}$ to
      \eqref{E:SWE} in $L^p(\Omega)$.
		\item If $t> T_2$, then for all $x\in\R^d$, the series in \eqref{E:SecondM} diverges.
			Hence, there is no $L^2(\Omega)$-solution to \eqref{E:SWE} whenever $t>T_2$.
	\end{enumerate}
\end{theorem}

The existence of the solution in Theorem \ref{T:Main} is proved in Theorem \ref{T:iff}.
The matching upper and lower bounds in \eqref{E:Main} are proved in Theorems \ref{T:Upper} and \ref{T:Lowbound}, respectively.
The two limits \eqref{E:tMain} and \eqref{E:pMain} are direct consequences of \eqref{E:Main}.
Theorem \ref{T:Critical} is covered by part (iii) of both Theorems~\ref{T:p2} and~\ref{T:Upper}.
\bigskip

\begin{remark}
	\label{R:Dalang}
	It is known that {\em Dalang's condition}:
	\begin{equation}
		\label{E:Dalang}
		\int_{\bR^d}\frac{1}{1+|\xi|^2}\mu(d\xi)<\infty,
	\end{equation}
	is the necessary and sufficient condition for the existence and uniqueness of the stochastic heat equation:
	\begin{equation}
		\label{heat}
		\frac{\partial u}{\partial t}(t,x)=\frac{1}{2}\Delta u(t,x)+u(t,x)\dot{W}(x),\ t>0,x\in \bR^d; \quad u(0,x)=1,
	\end{equation}
  with the same noise $\dot{W}$ as above. This is due to the fact that the
  Fourier transform of the heat kernel is nonnegative.  In contrast, the
  oscillatory nature of the Fourier transform of the wave kernel (see
  \eqref{E:FG}) makes \eqref{E:Dalang} only a sufficient condition.  Note that
  under Assumption \ref{A:Main}, Dalang's condition \eqref{E:Dalang} becomes
  $\alpha\in (0,2\wedge d)$.  The existence and uniqueness part of Theorem
  \ref{T:Main} leverages this oscillatory property and hence requires weaker
  conditions than \eqref{E:Dalang}.
\end{remark}

\begin{remark}
	\label{R:asyHeat}
  Similar to part (i) of Theorem \ref{T:Main}, under Assumption \ref{A:Main},
  for all $d\ge 1$ and $\alpha\in (0,2\wedge d)$,  it has been be proved in Chen
  \cite{xchen17} and \cite{xchen19} that the solution to equation \eqref{heat}
  satisfies: for any $p\geq 2$,
	\begin{equation}
		\label{E:asyH1}
		\lim_{t_p'\to \infty} (t_p')^{-\frac{4-\alpha}{2-\alpha}} \log \|u(t,x)\|_p
			=\theta^{\frac{2}{2-\alpha}}
			\frac{2-\alpha}{4} \left(\frac{4\cM}{4-\alpha} \right)^{\frac{4-\alpha}{2-\alpha}},
	\end{equation}
	where $t_p'=t(p-1)^{2/(4-\alpha)}$. In particular, by freezing $p\geq 2$ in \eqref{E:asyH1}, we have that:
	\begin{equation}
		\label{E:asyHeat}
		\lim_{t\to \infty} t^{-\frac{4-\alpha}{2-\alpha}} \log \E \left[u(t,x)^p\right]
		=p (p-1)^{\frac{2}{2-\alpha}} \theta^{\frac{2}{2-\alpha}}
		\frac{2-\alpha}{4} \left(\frac{4\cM}{4-\alpha} \right)^{\frac{4-\alpha}{2-\alpha}}.
	\end{equation}
  Here we point out that the moment asymptotics in both \cite{xchen17} and
  \cite{xchen19} mainly target on the parabolic Anderson equations with
  time-dependent Gaussian noise.  Given the fact that the approach in
  \cite{xchen17} and \cite{xchen19} (and in other recent papers on the
  asymptotics in the parabolic setting) is based on the Feynman-Kac moment
  representation
	\begin{equation*}
		\E\left(\vert u(t,x)\vert^p\right)=\E\left(\exp\bigg\{\theta^2\sum_{1\le j<k\le p}
		\int_0^t\!\int_0^t\gamma_0(s-r)\gamma\big(B_j(s)-B_k(r)\big)dsdr\bigg\}\right),
	\end{equation*}
	the time-independent case is often viewed as the special case, as far as
	moment intermittency is concerned, of
	the time-dependence setting by taking the time-covariance $\gamma_0(\cdot)=1$.
\end{remark}

\begin{remark}
	\label{R:Blowup}
  In the case of the stochastic heat equation with more general noises including
  the noise that we study here, characterizing the sharp blowup time for the
  $p$-th moment at the critical case has been recently studied by Chen et al
  \cite{CDOT20}; see Theorem 3.14, ibid.  In the case of the stochastic wave
  equation, the difficulty is the lack of the Feynman-Kac representation of the
  moments. Hence, here we are only able to establish the sharp transition for
  the second moment instead of the general $p$-th moment, $p\ge 2$.  We
  conjecture that part (ii) of Theorem \ref{T:Critical} is true for all $p\ge
  2$, that is, $T_p$ in \eqref{E:Tp} is the sharp transition time for the $p$-th
  moment for $p\ge 2$ in the sense that
	\begin{align}
		\label{E:Conjecture}
		\begin{cases}
			\displaystyle	\limsup_{N\to\infty} \Norm{u_N(t,x)}_p < \infty  & \text{if $t<T_p(\theta)$},\\[0.5em]
			\displaystyle	\liminf_{N\to\infty} \Norm{u_N(t,x)}_p = \infty & \text{if $t>T_p(\theta)$},
		\end{cases}
	\end{align}
	with $u_N(t,x) := 1+\sum_{n=1}^{N} \theta^{n/2}I_n\left(f_n(\cdot,x;t)\right)$; see \eqref{E:Series} for the notation.
	Note that for any $N<\infty$ fixed, $\Norm{u_N(t,x)}_p<\infty$ all $t>0$ and $x\in\R^d$ (see Lemma \ref{L:LpalaceNorm}).
\end{remark}

\begin{remark}
  The following representation plays a fundamental role in this paper thanks to
  the decomposition \eqref{E:gamma=K*K}:
	\begin{align}\begin{aligned}
		\label{E:keyNorm}
		\Norm{\widetilde{f}_n(\cdot,0;t)}_{\cH^{\otimes n}}^{2}
		 =&  \frac{1}{(n!)^2} \int_{(\bR^d)^n} \bigg[ \sum_{\sigma \in \Sigma_n}\int_{[0,t]_{<}^n}\int_{(\bR^d)^n}\\
		  &  \times \prod_{k=1}^{n}K(x_k-y_k) \prod_{k=1}^{n} G(s_{k}-s_{k-1},y_{\sigma(k)}-y_{\sigma(k-1)}) d{\bf y} d{\bf s}\bigg]^2 d{\bf x},
	\end{aligned}\end{align}
  where $\widetilde{f}_n(\cdot,x;t)$ is the kernel appearing in the Wiener chaos
  expansion of the solution, $\Sigma_n$ is the set of permutations of
  $1,\ldots,n$, and we use convention that $s_0=0$, and $y_{\sigma(0)}=0$.  Note
  that the integral in \eqref{E:keyNorm} needs to be properly handled when $d=3$
  because $G$ is a measure, which has been done in this paper through
  mollification.
\end{remark}

\begin{remark}
There are two approaches for the study of SPDEs driven by Gaussian noise: the Malliavin calculus approach based on the It\^o-Skorohod integral, and the pathwise approach, which is related to the Stratonovich integral and uses the regularity of the noise. In the case of the heat equation, it is possible to unify these approaches using the Feynman-Kac formula (see \cite{HHNT}). A similar study is not available at the moment for the wave equation. Therefore, it is unclear if the Skorohod solution that we investigate here coincides with the pathwise solution constructed using the connection with the Anderson Hamiltonian (see \cite{labbe13,labbe19,GUZ}). The relationship between the Skorohod solution and the pathwise solution is a different topic which may be treated in a separate project. An example of a wave equation whose solution exhibits finite-time blow-up can be found in \cite{ORSW}.
\end{remark}

\bigskip
Throughout this paper, the letter $\alpha$ is reserved for the scaling index in
\eqref{E:scaleGamma} and $\beta$ for
\begin{align}
	\label{E:beta}
	\beta :=
	\begin{cases}
		\displaystyle	\frac{4-\alpha}{3-\alpha} & \text{under Assumption \ref{A:Main};}\\[0.5em]
    \displaystyle	\frac{4-d}{3-d} & \text{under Assumption \ref{A:White}.}
	\end{cases}
\end{align}
Recall that in any dimension $d\geq 1$, the Fourier transform of the fundamental solution is
\begin{align}
	\label{E:FG}
	\cF G(t,\cdot)(\xi)=\frac{\sin(t|\xi|)}{|\xi|},\qquad \text{for all $\xi\in \mathbb{C}^d$ },
\end{align}
which satisfies the following scaling property:
\begin{equation}
	\label{E:scaleFG}
	\cF G(t,\cdot)(c\xi)=\frac{1}{c}\cF G(ct,\cdot)(\xi) \quad \mbox{for all} \ c>0.
\end{equation}

This article is organized as follows.
In Section \ref{S:Pre}, we introduce some basic elements of Malliavin calculus (see Nualart \cite{nualart06}),
and present some large deviation results that will be used in this paper.
Then in Section \ref{S:Solv}, we obtain solvability conditions.
The moment asymptotics for $p=2$ in \eqref{E:tMain} is proved in Section \ref{S:p=2}.
The upper and lower bounds for the asymptotics in \eqref{E:tMain} are proved in Sections \ref{S:upper} and \ref{S:lower},
respectively. Finally, one can find some auxiliary results in the appendices.

\section{Some preliminaries}
\label{S:Pre}
In this section, we present some preliminaries and set up some notation.

\subsection{Elements of Malliavin calculus}
\label{SS:Malliavin}

Let $\cH$ be the completion of $\cD(\bR^d)$ with respect to $\langle \cdot,\cdot \rangle_{\cH}$, where $\cD(\bR^d)$ is the space of infinitely differentiable functions on $\bR^d$ with compact support.
The map $\varphi \mapsto W(\varphi)$ is an isometry from $\cD(\bR^d)$ to $L^2(\Omega)$, which can be extended to $\cH$.
Then $\{W(\varphi);\varphi\in \cH\}$ is an isonormal Gaussian process,
i.e. $\E[W(\varphi)W(\psi)]=\langle \varphi,\psi \rangle_{\cH}$ for any $\varphi,\psi \in \cH$.

The space $\cH$ contains distributions $f\in \cS'(\bR^d)$ such that $\cF f \in L_{\bC}^2(\mu)$.
Let $|\cH|$ be the set of measurable functions $\varphi:\bR^d \to \bR$ such that $\InPrd{|\varphi|*\gamma, |\varphi|} <\infty$.
It is known that $|\cH| \subset \cH$.
If $\mu$ is a constant multiple of the Lebesgue measure, $\cH$ coincides with $L^2(\bR^d)$.

We denote by $\delta$ the \textit{Skorohod integral} with respect to $W$,
and by ${\rm Dom} \, \delta$ its domain. For any $u \in {\rm Dom}\, \delta$, we write
$$\delta(u)=\int_{\bR^d}u(x) W(\delta x).$$
By definition, $\delta(u)$ is a random variable in $L^2(\Omega)$.
We refer the reader to Section 1.3 of \cite{nualart06} for the precise definition of $\delta$.

\begin{definition}
	\label{D:Sol}
  A random field $u=\{u(t,x);t\geq 0,x\in \bR^d\}$ such that
  $\Norm{u(t,x)}_2<\infty$ for all $t\ge 0$ and $x\in\R^d$ is called {\em a mild
  solution} to equation \eqref{E:SWE} if the random field $\{G(t-s,x-y)u(s,y):\:
s\in[0,t],\: y\in\R^d\}$ is Skorohod integrable for each $s\in[0,t]$ fixed and
satisfies the following stochastic integral equation:
	\begin{equation}
		\label{E:Sol}
		u(t,x)=1+\sqrt{\theta} \int_0^t \left( \int_{\bR^d}G(t-s,x-y)u(s,y) W(\delta y)\right)ds.
	\end{equation}
\end{definition}
\bigskip

Intuitively, if the solution exists, it should be given by the series:
\begin{align}
	\label{E:Series}
	u(t,x)=1+\sum_{n\geq 1}\theta^{n/2}I_n(f_n(\cdot,x;t)),
\end{align}
where $I_n:\cH^{\otimes n} \to \cH_n$ is the multiple integral with respect to
$W$ and $\cH_n$ is the $n$-th Wiener chaos space.  If $d\leq 2$, the kernel
$f_n(\cdot,x;t)$ is an integrable function on $(\bR^d)^n$ given by
\begin{align}
\begin{aligned}
	\label{E:fn}
		f_n(x_1,\ldots,x_n,x;t) & =\int_{[0,t]_{<}^n}G(t-t_n,x-x_n)\ldots G(t_2-t_1,x_2-x_1)d\mathbf{t} \\
                            & =\int_{[0,t]_{<}^n}G(s_1,x_n-x)\ldots G(s_n-s_{n-1},x_1-x_2)d\mathbf{s},
\end{aligned}
\end{align}
where $[0,t]_{<}^n=\{(t_1,\ldots,t_n) \in [0,t]^n;t_1<\ldots<t_n\}$.
It can be shown that $f_n(\cdot,x;t)$ is finite.

If $d=3$, $f_n(\cdot,x;t)$ is a finite measure on $(\bR^3)^n$, given by:
\begin{align*}
	f_n(\cdot,x;t)=\int_{[0,t]_{<}^n}G(t-t_n,x-dx_n)G(t_n-t_{n-1},x_n-dx_{n-1})\ldots G(t_2-t_1,x_2-dx_1)d \mathbf{t}.
\end{align*}
We denote by $\widetilde{f}_n(\cdot,x;t)$ the symmetrization of $f_n(\cdot,x;t)$ in the variables $x_1,\ldots,x_n$. If $d\leq 2$,
\begin{align*}
	\widetilde{f}_n(x_1,\ldots,x_n,x;t)
	&=\frac{1}{n!}\sum_{\rho \in \Sigma_n} f_n(x_{\rho(1)},\ldots,x_{\rho(n)},x;t) \\
	&=\frac{1}{n!}\sum_{\rho \in \Sigma_n} \int_{[0,t]_{<}^n} G(t-t_n,x-x_{\rho(n)})\ldots G(t_2-t_1,x_{\rho(2)}-x_{\rho(1)}) d \mathbf{t}
\end{align*}
and in particular, for $x=0$,
\begin{equation}
\label{fn}
\widetilde{f}_n(x_1,\ldots,x_n,0;t)=\frac{1}{n!}\sum_{\sigma \in \Sigma_n} \int_{[0,t]_{<}^n} G(s_1,x_{\sigma(1)})\ldots G(s_n-s_{n-1},x_{\sigma(n)}-x_{\sigma(n-1)}) d \mathbf{s}.
\end{equation}

If $d=3$, $\widetilde{f}_n(\cdot,x;t)$ is a
measure with compact support. If $d\geq 4$, $f_n(\cdot,x;t)$ is a distribution whose Fourier transform can
be obtained via,
by setting $t_{n+1}=t$,
\begin{align}
	\label{E:Ffn}
	\cF f_n(\cdot,x;t)(\xi_1,\ldots,\xi_n)
	& =\e^{-i(\sum_{j=1}^n \xi_j)\cdot x}\int_{[0,t]_{<}^n} \prod_{k=1}^{n}\overline{\cF G(t_{k+1}-t_k,\cdot)\left(\sum_{j=1}^k\xi_j\right)} d \mathbf{t}.
\end{align}
Relation \eqref{E:Ffn} is true for any dimension $d$.
\bigskip

Setting $s_0=0$, we have:
\begin{align*}
\cF \widetilde{f}_n(\cdot,x;t)(\xi_1,\ldots,\xi_n)
	& =\e^{-i(\sum_{j=1}^n \xi_j)\cdot x}\sum_{\sigma \in \Sigma_n}\int_{[0,t]_{<}^n} \prod_{j=1}^{n}\overline{\cF G(s_{j}-s_{j-1},\cdot)\left(\sum_{k=j}^{n}\xi_{\sigma(k)}\right)} d \mathbf{s}.
\end{align*}

The following result gives the existence and uniqueness of solution. Its proof follows by a standard procedure.

\begin{theorem}
	\label{T:ExUn}
Let $d\geq 1$ be arbitrary.
  Fix any $T\in (0,\infty]$.  Suppose that $f_n(\cdot,x;t) \in \cH^{\otimes n}$
  for any $t\in (0,T)$, $x\in\R^d$ and $n\geq 1$.  Then equation \eqref{E:SWE}
  has a unique $L^2(\Omega)$-solution on $(0,T)\times\R^d$ if and only if the
  series in \eqref{E:Series} converges in $L^2(\Omega)$ for any
  $(t,x)\in(0,T)\times\R^d$, which is equivalent to the series in
  \eqref{E:ExUn2m} converges.  In this case, the solution is given by
  \eqref{E:Series} with the second moment given by
	\begin{align}
		\label{E:ExUn2m}
		\E\left(|u(t,x)|^2\right)
		=\sum_{n\geq 0} \theta^n n! \Norm{\widetilde{f}_n(\cdot,x;t)}_{\cH^{\otimes n}}^2, \quad \text{for any $(t,x)\in (0,T)\times\R^d$}.
	\end{align}
In particular, the $L^2(\Omega)$-solution exists under Dalang's condition:
\[
\int_{\bR^d}\frac{1}{1+|\xi|^2}\mu(d\xi)<\infty.
\]
\end{theorem}

\subsection{Some large deviation results}

In Chen \cite{xchen07} and Bass et al \cite{BCR09}, it has been proved that
\begin{equation}
	\label{BCR0}
	\lim_{n \to \infty}\frac{1}{n} \log \frac{1}{(n!)^2}
        \int_{(\bR^d)^n} \left[ \sum_{\sigma \in \Sigma_n}\prod_{k=1}^n \frac{1}{1+|\sum_{j=k}^n\xi_{\sigma(j)}|^2}\right]^2
	\mu(d\xi_1) \ldots \mu(d\xi_n) =\log \rho,
\end{equation}
where
\begin{align}
	\label{E:rho}
	\rho=\sup_{\Norm{h}_{L^2(\R^d)}=1} \int_{\bR^d}
	\left[\int_{\bR^d} \frac{h(\xi+\eta)h(\eta)}{\sqrt{1+|\xi+\eta|^2}\sqrt{1+|\eta|^2}} d\eta\right]^2 \mu(d\xi).
\end{align}
More precisely, the setting of
$\mu(d\xi)=(2\pi)^{-d}d\xi$ (under Assumption \ref{A:White}) is obtained in \cite[(2.1) and Theorem 2.1]{xchen07} with $p=2$ and
$\psi(\cdot)=\vert\cdot\vert^2$.
In the  subsequent paper \cite{BCR09}, \eqref{BCR0} was established for $\gamma(x)=\vert x\vert^{-\alpha}$ with $\alpha \in (0,d)$ (or
$\mu(d\xi)=C_{d,\alpha}\vert\xi\vert^{-(d-\alpha)}d\xi$, where $C_{d,\alpha}$ is given by \eqref{E:constCBeta});
 see Theorem 2.3 and Lemma 2.2 in \cite{BCR09}\footnote{There is small error in the statement of Lemma 2.2 of \cite{BCR09}: in relation (2.6) of \cite{BCR09}, in the argument of $Q$, one has $\sum_{j=k}^n$ instead of $\sum_{j=1}^k$.}
The proof of \eqref{BCR0} under the setting of \cite{BCR09} can be easily adapted to the setting described in Assumption \ref{A:Main}.

By Theorem 1.5 of \cite{BCR09}, one can see that\footnote{There is a small error in Theorem 1.5 of Bass et al \cite{BCR09}
	which states that $\rho=(2\pi)^{-d}\Lambda_{\alpha}^{2-\alpha/2}$. The correct result is $\rho=\Lambda_{\alpha}^{2-\alpha/2}$.
	This error is due to the fact that in the first line of equation (7.22), ibid.,
	one should have $(2\pi)^{-dp}$ instead of $(2\pi)^{-d(p+1)}$,
	since $\cF (f^p)=(2\pi)^{-d(p-1)}(\cF f)^{*p}$.} $\rho= \Lambda_{\gamma}^{(4-\alpha)/2}$ where
      $\Lambda_\gamma = \cM(\gamma,2)$.
These results play an important role in this paper.
\medskip

The functional $\cM(\cdot,\cdot)$ has the following scaling property:
\begin{lemma}
	\label{L:scaleM}
	Let $f$ be an arbitrary nonnegative definite measure such that for some $\alpha< 4$,
	$f\left(c A\right)= c^{d-\alpha}f(A)$ for all $c>0$ and $A\in \mathcal{B}(\R^d)$.
	In case $f$ has a density, this scaling property becomes $f(cx)=c^{-\alpha} f(x)$ for all $x\in\R^d$.
	Then
	\begin{align}
		\label{E:scaleM}
		\cM\left(\Theta f, \theta \right)
		= \Theta^{\frac{2}{4-\alpha}} \theta^{-\frac{\alpha}{4-\alpha}}  \cM(f,1),
		\quad\text{for all $\Theta>0$ and $\theta>0$.}
	\end{align}
	In particular, when $f=\delta_0$ and $d<4$, \eqref{E:scaleM} still holds with $\alpha$ replaced by $d$.
\end{lemma}

This lemma is proved under slightly different settings in Lemmas B.2 and B.3 of Balan and Song \cite{balan-song19}.
We give the proof below for readers' convenience.

\begin{proof}
	For any $\theta,\Theta>0$, by the definition of $\cM$ in \eqref{E:cM}, we see that
	\begin{align*}
		\cM\left(\Theta f,\theta\right)
		&=\sup_{g \in \cF_d} \left\{ \InPrd{g^2* \Theta f,g^2}^{1/2}
			-\frac{\theta}{2} \int_{\bR^d}|\nabla g(x)|^2 dx \right\}\\
		&=\theta\sup_{g \in \cF_d} \left\{ \frac{\sqrt{\Theta}}{\theta}
			\InPrd{g^2*f,g^2}^{1/2}
			-\frac{1}{2} \int_{\bR^d}|\nabla g (x)|^2 dx \right\}.
	\end{align*}
	Let $g\in W^{1,2}(\R^d)$.
	For any $B>0$, by defining $\widetilde{g}(x) = B^{-d/2}g\left(B^{-1}x\right)$,
	one sees that the $\Norm{\cdot}_{L^2(\R^d)}$ norm is preserved, i.e., $\Norm{g}_{L^2(\R^d)} = \Norm{\widetilde{g}}_{L^2(\R^d)}$.
	Notice that $\cF \widetilde{g}(\xi) = B^{d/2}\cF g\left(B\xi\right)$, which implies that
	$\int_{\R^d}|\xi|^2 \left|\cF \widetilde{g} (\xi)\right|^2 d\xi = B^{-2} \int_{\R^d}|\xi|^2 \left|\cF
	g(\xi)\right|^2d\xi$, or equivalently,
	\begin{align*}
		\int_{\R^d}\left|\nabla \widetilde{g} (x)\right|^2 dx
		=B^{-2} \int_{\R^d}\left|\nabla g(x)\right|^2 dx.
	\end{align*}
	By the scaling property of $f$, we see that $\left(\widetilde{g}^2 *f \right)(x) =B^{-\alpha}
	\left(g^2*f\right)\left(B^{-1}x\right)$ and hence
	$\InPrd{\widetilde{g}^2*f,\widetilde{g}^2 } = B^{-\alpha} \InPrd{g^2*f,g^2}$.
	Combining these relations, we see that
	\begin{align*}
		\cM\left(\Theta f,\theta\right)
		&=\theta\sup_{\widetilde{g} \in \cF_d} \left\{
			\frac{\sqrt{\Theta}}{\theta} B^{\alpha / 2}
			\InPrd{\widetilde{g}^2*f,\widetilde{g}^2}^{1/2}
			-\frac{B^{2}}{2}  \int_{\bR^d}|\nabla \widetilde{g} (x)|^2 dx \right\}.
	\end{align*}
	Hence, by making the following choice of $B$
	\begin{align*}
		\frac{\sqrt{\Theta}}{\theta} B^{\alpha / 2} 	= B^2 \quad \Longleftrightarrow \quad
		B = \left(\frac{\sqrt{\Theta}}{\theta}\right)^{\frac{2}{4-\alpha}}\quad\text{and} \quad \alpha\ne 4,
	\end{align*}
  we see that $\cM$ has the desired scaling property \eqref{E:scaleM}.
\end{proof}

By the scaling property of $\cM$, namely, $\cM(\gamma,2)=\left(1/2 \right)^{\frac{\alpha}{4-\alpha}}\cM\left(\gamma,1\right)$,
we have the following important relation:
\begin{align}
	\label{E:rhoM}
	\rho=\left[\left(1/2\right)^{\frac{\alpha}{4-\alpha}}\cM\right]^{(4-\alpha)/2}=\left(1/2
	\right)^{\alpha/2}\cM^{(4-\alpha)/2}.
\end{align}
In the white noise case, one simply replaces $\alpha$ and $\cM$ in \eqref{E:rhoM} by $d$ and $\cM(\delta_0)$, respectively.

\section{Solvability}
\label{S:Solv}

The aim of this section is to prove the following theorem:

\begin{theorem}
	\label{T:iff}
	(i) Suppose that either Assumption \ref{A:Main} holds with $0<\alpha<d\leq 3$ or Assumption \ref{A:White} holds.
	Then equation \eqref{E:SWE} has a unique solution $\{u(t,x);t>0,x\in \bR^d\}$ and
	\begin{equation}
		\label{E:mom-p}
		\sup_{(t,x)\in [0,T] \times \bR^d}\bE\left(|u(t,x)|^p\right)<\infty  \quad \mbox{for all $p\geq 2$ and $T>0$}.
	\end{equation}
	(ii) Under Assumption \ref{A:Critical}, for any $p\ge 2$, equation \eqref{E:SWE} is well defined with finite $p$-th moments for all $t\in (0,T_p')$ and $x\in\R^d$, where
	\begin{align}
		\label{E:Tp'}
		T_p'  = \frac{4\pi}{\theta(p-1)}.
	\end{align}
\end{theorem}

We need some preparations before proving this theorem. Let us first introduce some notation. Recall equation \eqref{fn}.
For $d\le 2$, set
\begin{align}
	\nonumber
	L_n(y_1,\ldots,y_n)&=n!\int_0^{\infty} \e^{-t} \widetilde{f}_{n}(y_1,\ldots,y_n,0;t)dt \\
	\label{E:Ln}
	&=\sum_{\sigma \in \Sigma_n} \int_{0}^{\infty} \e^{-t} \int_{[0,t]_{<}^n} \prod_{k=1}^{n}G(s_k-s_{k-1},y_{\sigma(k)}-y_{\sigma(k-1)})d\bs dt.
\end{align}
If $d\leq 2$ and Assumption \ref{A:Main} holds, set for any $t>0$ and ${\bf x}=(x_1,\ldots,x_n)\in (\bR^d)^n$,
\begin{align}
	\nonumber
	H_n(t,{\bf x})&=n! \int_{(\bR^d)^n} \prod_{k=1}^{n}K(x_k-y_k) \widetilde{f}_n(y_1,\ldots,y_n,0;t) d{\bf y} \\
	\label{E:Hn}
	&=\sum_{\sigma \in \Sigma_n}\int_{[0,t]_{<}^n}\int_{(\bR^d)^n}\prod_{k=1}^{n}K(x_k-y_k) \prod_{k=1}^{n} G(s_{k}-s_{k-1},y_{\sigma(k)}-y_{\sigma(k-1)}) d{\bf y} d{\bf s},
\end{align}
where for the last line we used relation \eqref{fn} for expressing $\widetilde{f}_n(\cdot;0,t)$.
If Assumption \ref{A:White} holds, set for any $t>0$ and ${\bf x}=(x_1,\ldots,x_n)\in (\bR^d)^n$
\begin{align*}
  H_n(t,{\bf x})=n! \widetilde{f}_n(x_1,\ldots,x_n,0;t)=\sum_{\sigma \in \Sigma_n}\int_{[0,t]_{<}^n}  \prod_{k=1}^{n} G(s_{k}-s_{k-1},x_{\sigma(k)}-x_{\sigma(k-1)}) d{\bf s},
\end{align*}
with $s_0=0$, $x_{\sigma(0)}=0$.
Similarly, if $d=3$, for any $\varepsilon>0$, set
\begin{align}
	\label{E:Lne}
	&L_{n,\varepsilon}(y_1,\ldots,y_n)
	=\sum_{\sigma \in \Sigma_n} \int_{0}^{\infty} \e^{-t} \int_{[0,t]_{<}^n} \prod_{k=1}^{n}G_{\varepsilon}(s_k-s_{k-1},y_{\sigma(k)}-y_{\sigma(k-1)})d\bs dt,
\end{align}
where $G_{\varepsilon}(t,\cdot)=G(t,\cdot)*p_{\varepsilon}$ with $p_{\varepsilon}(x)=(2\pi\varepsilon)^{-d/2}\e^{-|x|^2/(2\varepsilon)}$.
Moreover, if $d\leq 3$, under Assumption \ref{A:Main} and \ref{A:Critical}, denote respectively
\begin{align} \begin{aligned}
	\label{E:Hne}
	&H_{n,\varepsilon}(t,{\bf x})=
	\sum_{\sigma \in
	\Sigma_n}\int_{[0,t]_{<}^n}\int_{(\bR^d)^n}\prod_{k=1}^{n}K(x_k-y_k)
	\prod_{k=1}^{n} G_{\varepsilon}(s_{k}-s_{k-1},y_{\sigma(k)}-y_{\sigma(k-1)}) d{\bf y} d{\bf s}, \\
	&H_{n,\varepsilon}(t,{\bf x})=
		\sum_{\sigma \in \Sigma_n}\int_{[0,t]_{<}^n} \prod_{k=1}^{n} G_{\varepsilon}(s_{k}-s_{k-1},x_{\sigma(k)}-x_{\sigma(k-1)}) d{\bf s}.\\
\end{aligned} \end{align}

\begin{remark}
	\label{R:Ge}
	A key observation is that the mollification $G_\varepsilon\left(t,x\right)$ is a nonnegative function if and only if $d\le 3$. When $d\le 2$,
	$G$ is a function and there is no need for this mollification.
  We only use this mollification in the case $d=3$. In this case, $G$ is a locally finite measure on $\R_+\times\R^3$ that is singular with
	respect to the Lebesgue measure.
	Because the nonnegativity property of $G_\varepsilon\left(t,x\right)$ is not valid for $d\ge 4$,
	the results presented in this section cover only the case $d\le 3$.
\end{remark}

If $d\leq 2$, denote
\begin{align*}
	L(x)       = \int_0^{\infty}\e^{-t}G(t,x)dt \quad \mbox{and} \quad
	\cF L(\xi) = \int_{0}^{\infty}\e^{-t}\cF G(t,\cdot)(\xi)dt
	           = \frac{1}{1+|\xi|^2},
\end{align*}
and if $d=3$, set
\begin{align*}
	L_\varepsilon(x)       = \int_0^{\infty}\e^{-t}G_\varepsilon(t,x)dt \quad \mbox{and} \quad
	\cF L_\varepsilon(\xi) = \int_{0}^{\infty}\e^{-t}\cF G_\varepsilon(t,\cdot)(\xi)dt
	                       = \frac{\e^{-\left(\varepsilon/2\right)|\xi|^2}}{1+|\xi|^2}.
\end{align*}
By change of variables, one sees that the Laplace transform of convolutions (in time) becomes product:
\begin{align*}
	L_n(y_1,\ldots,y_n)
	&= \sum_{\sigma \in S_n}L(y_{\sigma(1)}) L(y_{\sigma(2)}-y_{\sigma(1)})\ldots L(y_{\sigma(n)}-y_{\sigma(n-1)}),
	&\text{$d\le 2$,}\\
	L_{n,\varepsilon}(y_1,\ldots,y_n)
	&= \sum_{\sigma \in S_n}L_\varepsilon(y_{\sigma(1)}) L_\varepsilon(y_{\sigma(2)}-y_{\sigma(1)})\ldots L_\varepsilon(y_{\sigma(n)}-y_{\sigma(n-1)}),
	&\text{$d=3$.}
\end{align*}
Hence,
\begin{equation}\begin{aligned}
	\label{E:FLn}
	\cF L_n(\xi_1,\ldots,\xi_n)
	&= \sum_{\sigma \in \Sigma_n} \prod_{k=1}^{n} \frac{1}{1+|\sum_{j=k}^n \xi_{\sigma(j)}|^2}, & \text{$d\le 2$,}\\
	\cF L_{n,\varepsilon}(\xi_1,\ldots,\xi_n)
	&= \sum_{\sigma \in \Sigma_n} \e^{-\frac{\varepsilon}{2} \sum_{k=1}^{n}|\xi_{\sigma(k)}+\ldots+\xi_{\sigma(n)}|^2}\prod_{k=1}^{n} \frac{1}{1+|\sum_{j=k}^n \xi_{\sigma(j)}|^2}, & \text{$d=3$.}
\end{aligned}\end{equation}

\begin{lemma}
	\label{L:IntHn}
	If $d\leq 2$ and either parts (i) and (iii) of Assumption \ref{A:Main} hold, or Assumption \ref{A:White} holds, then
	\begin{align}
		\label{E:IntHn}
		\int_{(\bR^d)^n} \left[\int_{0}^{\infty}\e^{-t}H_n(t,\bx) dt\right]^2 d\bx
		& =\int_{(\bR^d)^n} |\cF L_n(\xi_1,\ldots,\xi_n)|^2\mu(d\xi_1)\ldots \mu(d\xi_n).
	\end{align}
	If $d=3$ and either parts (i) and (iii) of Assumption \ref{A:Main} hold, or Assumption \ref{A:Critical} holds, then
	relation \eqref{E:IntHn} holds with $H_n$ and $\cF L_n$ replaced by
	$H_{n,\varepsilon}$ and $\cF L_{n,\varepsilon}$, respectively.
\end{lemma}

\begin{proof}
	It suffices to prove the result in the case when $d\leq 2$ and parts (i) and (iii) of Assumption \ref{A:Main} hold. The white noise case is proved in a similar way.
	Notice that by Fubini's theorem,
	\begin{align*}
		\int_{0}^{\infty}&\e^{-t}H_n(t,\bx) dt\\
		&=\sum_{\sigma \in \Sigma_n} \int_0^{\infty}\e^{-t} \int_{[0,t]_{<}^n} \int_{\bR^{nd}} \prod_{k=1}^{n}K(x_k-y_k) \prod_{k=1}^{n}G(s_k-s_{k-1},y_{\sigma(k)}-y_{\sigma(k-1)})d{\bf y} d{\bf s} d t\\
		&=\int_{(\bR^d)^n} \prod_{k=1}^n K(x_k-y_k) L_n(y_1,\ldots,y_n)d{\bf y}.
	\end{align*}
	Hence,
	\begin{align*}
		\int_{(\bR^d)^n} \left[\int_{0}^{\infty}\e^{-t}H_n(t,\bx) dt\right]^2 d\bx
		&=\int_{(\bR^d)^n} \left[ \int_{(\bR^d)^n} \prod_{k=1}^{n}K(x_k-y_k) L_{n}(y_1,\ldots,y_n)d\by\right]^2 d\bx \\
		&=\int_{(\bR^d)^n}|\cF L_n(\xi_1,\ldots,\xi_n)|^2 \mu(d\xi_1)\ldots \mu(d\xi_n),
	\end{align*}
	where the last step is due to the Plancherel theorem and the fact that $K*K=\gamma$.
\end{proof}

Next we prove the key expression \eqref{E:keyNorm} which was announced in the introduction.

\begin{lemma}
	If $d\leq 2$ and either parts (i) and (iii) of Assumption \ref{A:Main} hold or Assumption \ref{A:White} holds, then
	\begin{equation}
		\label{norm-fn-X}
		\Norm{\widetilde{f}_n(\cdot,0;t)}_{\cH^{\otimes n}}^{2}=\frac{1}{(n!)^2} \int_{(\bR^d)^n}H_n^2(t,{\bf x}) d{\bf x}.
	\end{equation}
\end{lemma}

\begin{proof}
	We treat only the case of Assumption \ref{A:Main}. The white noise case can be proved in a similar way.
	Writing $\gamma(y_k-y_k')=\int_{\bR^d}K(y_k-x_k)K(y_k'-x_k)dx_k$, we see that
	\begin{align*}
		\Norm{\widetilde{f}_n(\cdot,0;t)}_{\cH^{\otimes n}}^2
		&=\int_{(\bR^{d})^n} \int_{(\bR^d)^n}\prod_{k=1}^{n}\gamma(y_k-y_k') \widetilde{f}_n(y_1,\ldots,y_n,0;t)\widetilde{f}_n(y_1',\ldots,y_n',0;t)d{\bf y} d{\bf y}' \\
		&= \int_{(\bR^d)^n} \left[ \int_{(\bR^d)^n} \prod_{k=1}^{n}K(y_k-x_k) \widetilde{f}_n(y_1,\ldots,y_n,0;t)d{\bf y} \right]^2 d{\bf x}\\
		&=\int_{(\bR^d)^n} \left[ \frac{1}{n!} H_{n}(t,{\bf x})\right]^2 d{\bf x}.
	\end{align*}
\end{proof}

If $f$ is a non-decreasing function on $[0,\infty)$, then the following
{\it reverse Cauchy-Schwarz inequality} holds, which plays a key role in our developments.

\begin{lemma}
	\label{L:NonDecr}
	If $f: [0,\infty)\to [0,\infty)$ is a non-decreasing function, then
	\begin{align*}
		2 \int_{0}^{\infty} \e^{-2t}f^2(t)dt\leq \left( \int_0^{\infty}\e^{-t}f(t)dt\right)^2.
	\end{align*}
\end{lemma}

\begin{proof}
	Let $\tau$ and $\tau'$ be two independent random variables with exponential distribution of mean $1$.
	Then $\tau \wedge \tau'$ has an exponential distribution with density $g(t)=2 \e^{-2t}$ and
	\begin{align*}
		\left(\int_0^{\infty} \e^{-t} f(t) dt \right)^2
		=\E[f(\tau)]E[f(\tau')]
		=\E[f(\tau) f(\tau')]
		\geq E[f^2(\tau \wedge \tau')]
		=2\int_0^{\infty}\e^{-2t}f^2(t)dt.
	\end{align*}
\end{proof}

Denote
\begin{equation}
	\label{E:def-Tn}
	T_n:=\int_{(\bR^d)^n} \left[\sum_{\sigma \in \Sigma_n} \prod_{k=1}^{n} \frac{1}{1+|\sum_{j=k}^n \xi_{\sigma(j)}|^2} \right]^2  \mu(d\xi_1)\ldots \mu(d\xi_n).
\end{equation}

\begin{lemma}
	\label{L:Tn}
        Under the assumptions of Theorem \ref{T:Main}
        or Theorem \ref{T:Critical}, we have that
	\begin{align}
		\label{E:alpha<4}
		C_{\mu}':=\int_{\bR^d}\left(\frac{1}{1+|\xi|^2} \right)^2 \mu(d\xi)<\infty
		\quad\text{and}\quad
		T_n \leq (n!)^2 (C_{\mu}')^n.
	\end{align}
\end{lemma}

\begin{proof}
  Recall that $\varphi$ is the density function of $\mu$.
	By the scaling property of $\varphi$ in \eqref{E:scaleMu},
	$$
	\int_{\bR^d}\left(\frac{1}{1+|\xi|^2} \right)^2 \mu(d\xi)=
	\int_{\R^d}\bigg({1\over 1+\vert\xi\vert^2}\bigg)^2\varphi(\xi)d\xi
	=\int_{\R^d}\bigg({1\over 1+\vert\xi\vert^2}\bigg)^2\vert\xi\vert^{-(d-\alpha)}
	\varphi\Big({\xi\over\vert\xi\vert}\Big)d\xi.
	$$
	By the spherical substitution,
	\begin{align*}
		\int_{\R^d}\bigg({1\over 1+\vert\xi\vert^2}\bigg)^2\vert\xi\vert^{-(d-\alpha)} \varphi\Big({\xi\over\vert\xi\vert}\Big)d\xi
		&=\bigg(\int_{S_1}\varphi(\xi)\sigma(d\xi)\bigg)\int_0^\infty\big(1+\rho^2\big)^{-2}\rho^{-(d-\alpha)} \rho^{d-1}d\rho\\
		&=\bigg(\int_{S_1}\varphi(\xi)\sigma(d\xi)\bigg) \int_0^\infty\big(1+\rho^2\big)^{-2}\rho^{\alpha -1}d\rho,
	\end{align*}
	where $S_1$ is the unit sphere in $\R^d$ and $\sigma (d\xi)$ is the uniform sphere measure.
	On the other hand, let $B_1$ be the unit ball in $\R^d$.
	The fact that $\mu$ is tempered implies that
	$\mu(B_1)<\infty$. In addition,
	$$
	\mu(B_1)=\int_{B_1}\varphi(\xi)d\xi=\bigg(\int_{S_1}\varphi(\xi)\sigma(d\xi)
	\bigg)\int_0^1\rho^{-(d-\alpha)}\rho^{d-1}d\rho
	=\alpha^{-1}\int_{S_1}\varphi(\xi)\sigma(d\xi).
	$$
	Hence,
	\begin{align}
		\label{int}
		\int_{\bR^d}\left(\frac{1}{1+|\xi|^2} \right)^2 \mu(d\xi)=\alpha\mu(B_1)
		\int_0^\infty\big(1+\rho^2\big)^{-2}\rho^{\alpha -1}d\rho<\infty,
	\end{align}
	as $\alpha<4$. Therefore, the first relation in \eqref{E:alpha<4} holds as $\alpha<4$.
	\medskip

	As for the second relation in \eqref{E:alpha<4}, we have the following:
	\begin{align}
		T_n & \leq n! \int_{(\bR^d)^n} \sum_{\sigma \in \Sigma_n}
		\prod_{k=1}^{n} \left(\frac{1}{1+|\sum_{j=k}^n \xi_{\sigma(j)}|^2}\right)^2
		\mu(d\xi_1)\ldots \mu(d\xi_n) \notag \\
		& = (n!)^2 \int_{(\bR^d)^n}
		\prod_{k=1}^{n} \left(\frac{1}{1+|\sum_{j=k}^n \xi_{j}|^2}\right)^2
		\mu(d\xi_1)\ldots \mu(d\xi_n) \notag \\
		& \leq (n!)^2 \left[\sup_{\eta \in \bR^d} \int_{\bR^d}\left(\frac{1}{1+|\xi+\eta|^2}\right)^2
		\mu(d\xi)\right]^n \notag \\
                    & = (n!)^2 \left[ \int_{\bR^d}\left(\frac{1}{1+|\xi|^2} \right)^2 \mu(d\xi)\right]^n,
		\label{E_:alpha<4}
	\end{align}
	where the last line follows by Lemma 4.1 of Balan and Song \cite{balan-song17} or Lemma 3.1 of Chen \cite{xchen19} when Assumption A holds, and holds trivially when either
	Assumption B or C holds.
\end{proof}

In the case when Assumption \ref{A:Main} holds, by the scaling property \eqref{E:scaleFG} of $\cF G$,
	\begin{equation}
		\label{E:scaleTildeFn}
		\Norm{\widetilde{f}_n(\cdot,0;t)}_{\cH^{\otimes n}}^2 =t^{(4-\alpha)n}\Norm{\widetilde{f}_n(\cdot,0;1)}_{\cH^{\otimes n}}^2.
	\end{equation}

\begin{lemma}
	\label{L:LpalaceNorm}
	If Assumption \ref{A:Main} holds with $0<\alpha<d\le 3$, then
	\begin{equation}
		\label{E:LapaceNorm}
		\int_ 0^{\infty} \e^{-t}  \Norm{\widetilde{f}_n(\cdot,0;t)}_{\cH^{\otimes n}}^2 dt
		\leq 2^{(4-\alpha)n} \frac{1}{(n!)^2} T_n
		\leq \left(2^{4-\alpha} C_{\mu}'\right)^n
		<\infty,
	\end{equation}
	and consequently, $f_n(\cdot,0;t) \in \cH^{\otimes n}$ for any $t>0$.
	Relation \eqref{E:LapaceNorm} also holds under either Assumption \ref{A:White} or \ref{A:Critical} with $\alpha$ replaced by $d$.
\end{lemma}

\begin{proof}
	We first assume that Assumption \ref{A:Main} holds.
	Note that \eqref{E:LapaceNorm} implies that $\Norm{f_n(\cdot,0;t)}_{\cH^{\otimes n}}<\infty$ for almost all $t>0$.
	Due to \eqref{E:scaleTildeFn}, this has to hold for all $t>0$. Hence, $f_n(\cdot,0;t) \in \cH^{\otimes n}$ for all $t>0$.
	It remains to prove the first inequality in \eqref{E:LapaceNorm}. For this, we treat separately the cases $d\leq 2$ and $d=3$.
	\bigskip

	{\noindent\em Case 1. Assume that $d\leq 2$.}
	By the scaling property of $\cF G$ in \eqref{E:scaleFG}, we see that
	\begin{equation*}
		\cF \widetilde{f}_n(\cdot,0;ct)(\xi_1,\ldots,\xi_n)
		=c^{2n}\cF \widetilde{f}_n(\cdot,0;t)(c\xi_1,\ldots,c\xi_n), \quad \text{for all $c>0$}.
	\end{equation*}
	Using the above scaling property and the scaling property of $\mu$ (see \eqref{E:scaleMu}), we have:
	\begin{align*}
		\int_ 0^{\infty}\e^{-t}\Norm{\widetilde{f}_n(\cdot,0;t)}_{\cH^{\otimes n}}^2 dt
		&=\int_0^{\infty}\e^{-t} \int_{(\bR^d)^n} |\cF \widetilde{f}_n(\cdot,0;t)(\xi_1,\ldots,\xi_n)|^2  \mu(d\xi_1)\ldots \mu(d\xi_n)dt \\
		&=2^{(4-\alpha)n}  \int_0^{\infty}2 \e^{-2t}\Norm{\widetilde{f}_n(\cdot,0;t)}_{\cH^{\otimes n}}^2 dt \\
		&=2^{(4-\alpha)n}  \frac{1}{(n!)^2} \int_{(\bR^d)^n}\left(\int_0^{\infty}2 \e^{-2t}H_n^2(t,{\bf x})dt\right) d {\bf x}.
	\end{align*}
	where for the last line, we used \eqref{norm-fn-X}.
	Since the function $t \mapsto H_n(t,{\bf x})$ is non-negative and non-decreasing, we can apply Lemma
	\ref{L:NonDecr} to see that
	\begin{align*}
		\int_ 0^{\infty}\e^{-t}\Norm{\widetilde{f}_n(\cdot,0;t)}_{\cH^{\otimes n}}^2 dt
		&\leq 2^{(4-\alpha)n} \frac{1}{(n!)^2} \int_{(\bR^d)^n} \left(\int_0^{\infty}\e^{-t}H_n(t,{\bf x})dt \right)^2 d{\bf x} \\
		&= 2^{(4-\alpha)n} \frac{1}{(n!)^2} \int_{(\bR^d)^n} |\cF L_n(\xi_1,\ldots,\xi_n)|^2 \mu(d\xi_1)\ldots \mu(d\xi_n) \\
		&= 2^{(4-\alpha)n} \frac{1}{(n!)^2} T_n
	\end{align*}
	where for the second last line we used Lemma \ref{L:IntHn} and for the last line we used the form of $\cF L_n(\xi_1,\ldots,\xi_n)$ given by \eqref{E:FLn}.

	\bigskip

	{\noindent\em Case 2. Assume that $d=3$.}
	Consider the kernel
	\begin{equation}
		\label{def-gn}
		f_{n,\varepsilon}(y_1,\ldots,y_n,x;t)
		=\int_{[0,t]_{<}^n} \prod_{k=1}^{n}G_{\varepsilon}(t_{k+1}-t_k,y_{k+1}-y_k)d{\bf t},
	\end{equation}
	with $t_{n+1}=t$ and $y_{n+1}=x$. Then
	\begin{equation}
		\label{F-g-e}
		\cF f_{n,\varepsilon}(\cdot,x;t)(\xi_1,\ldots,\xi_n)
		=\exp\left(-\frac{\varepsilon}{2} \sum_{k=1}^{n}|\xi_1+\ldots+\xi_k|^2 \right)
		\cF f_n(\cdot,x;t)(\xi_1,\ldots,\xi_n).
	\end{equation}
	If $\widetilde{f}_{n,\varepsilon}(\cdot,x;t)$ be the symmetrization of $f_{n,\varepsilon}(\cdot,x;t)$, then
	\begin{align*}
		\cF \widetilde{f}_{n,\varepsilon}(\cdot,x;t) (\xi_1,\ldots,\xi_n)
		= & \frac{1}{n!}\sum_{\rho \in \Sigma_n} \exp\left(-\frac{\varepsilon}{2} \sum_{k=1}^{n}|\xi_{\rho(1)}+\ldots+\xi_{\rho(k)}|^2 \right) \\
	          & \times \cF f_n\left(\cdot,x;t\right)(\xi_{\rho(1)},\ldots,\xi_{\rho(n)}).
	\end{align*}
	Note that
	\begin{equation}
		\label{Fe-conv}
		\lim_{\varepsilon \to 0_+} \cF \widetilde{f}_{n,\varepsilon}(\cdot,x;t)(\xi_1,\ldots,\xi_n)
		=\cF \widetilde{f}_{n}(\cdot,x;t)(\xi_1,\ldots,\xi_n).
	\end{equation}
	The function $\cF \widetilde{f}_{n,\varepsilon}(\cdot,0;t)$ has the following scaling property:
	\begin{align*}
		\cF \widetilde{f}_{n,\varepsilon}(\cdot,0;ct)(\xi_1,\ldots,\xi_n)
		=c^{2n} \cF \widetilde{f}_{n,\:\varepsilon/c^2}(\cdot,0;t)(c\xi_1,\ldots,c\xi_n), \quad\text{for any $c>0$}.
	\end{align*}
	This leads to:
	\begin{align*}
		\Norm{\widetilde{f}_{n,\varepsilon}(\cdot,t)}_{\cH^{\otimes n}}^2
		= t^{(4-\alpha)n} \Norm{\widetilde{f}_{n,\varepsilon/t^2}(\cdot,0;1)}_{\cH^{\otimes n}}^2
	\end{align*}
	and
	\begin{align*}
		\int_ 0^{\infty}\e^{-t}\Norm{\widetilde{f}_{n,\varepsilon}(\cdot,0;t)}_{\cH^{\otimes n}}^2 dt
		=2^{(4-\alpha)n}  \int_0^{\infty}2 \e^{-2t}\Norm{\widetilde{f}_{n,\varepsilon/4}(\cdot,0;t)}_{\cH^{\otimes n}}^2 dt.
	\end{align*}

	Similar to \eqref{norm-fn-X}, by replacing $G$ by $G_{\varepsilon}$, we have:
	\begin{equation}
		\label{norm-fne}
		\Norm{\widetilde{f}_{n,\varepsilon}(\cdot,0;t)}_{\cH^{\otimes n}}^{2}
		=\frac{1}{(n!)^2} \int_{(\bR^d)^n}H_{n,\varepsilon}^2(t,{\bf x}) d{\bf x}.
	\end{equation}
	Since the function $t \mapsto H_{n,\varepsilon}(t,{\bf x})$ is non-negative and non-decreasing, by the same argument as in Case 1,
	we see that
	\begin{align*}
		\int_ 0^{\infty}\e^{-t}\Norm{\widetilde{f}_{n,\varepsilon}(\cdot,0;t)}_{\cH^{\otimes n}}^2 dt
		& =2^{(4-\alpha)n}  \frac{1}{(n!)^2} \int_{(\bR^d)^n}\left(\int_0^{\infty}2 \e^{-2t}H_{n,\varepsilon/4}^2(t,{\bf x})dt\right) d {\bf x}\\
		& \leq 2^{(4-\alpha)n} \frac{1}{(n!)^2} \int_{(\bR^d)^n} \left(\int_0^{\infty}\e^{-t}H_{n,\varepsilon/4}(t,{\bf x})dt \right)^2 d{\bf x} \\
    & = 2^{(4-\alpha)n} \frac{1}{(n!)^2}  \int_{(\bR^d)^n}|\cF L_{n,\varepsilon/4}(\xi_1,\ldots,\xi_n)|^2 \mu(d\xi_1) \ldots \mu(d\xi_n),
	\end{align*}
	where the last line is due to Lemma \ref{L:IntHn}.
	By bounding from above the exponential term in the expression of $\cF L_{n,\varepsilon}$ in \eqref{E:FLn} by $1$, we see that
	\begin{align*}
		\int_ 0^{\infty}\e^{-t}\Norm{\widetilde{f}_{n,\varepsilon}(\cdot,0;t)}_{\cH^{\otimes n}}^2 dt
		 & \leq 2^{(4-\alpha)n} \frac{1}{(n!)^2} \int_{(\bR^d)^n} \left[\sum_{\sigma \in \Sigma_n} \prod_{k=1}^{n} \frac{1}{1+|\sum_{j=k}^n \xi_{\sigma(j)}|^2} \right]^2
		 	\mu(d\xi_1)\ldots \mu(d\xi_n)\\
		 & = 2^{(4-\alpha)n} \frac{1}{(n!)^2} T_n.
	\end{align*}
	Using Fatou's lemma and \eqref{Fe-conv},
	\begin{align*}
		\int_ 0^{\infty}\e^{-t}\Norm{\widetilde{f}_{n}(\cdot,0;t)}_{\cH^{\otimes n}}^2 dt
		& =\int_0^{\infty}\e^{-t} \int_{(\bR^d)^n}|\cF \widetilde{f}_{n}(\cdot,0;t)(\xi_1,\ldots,\xi_n)|^2 \mu(d\xi_1) \ldots \mu(d\xi_n)\\
		& \leq \liminf_{\varepsilon \to 0+}\int_0^{\infty}\e^{-t} \int_{(\bR^d)^n}|\cF \widetilde{f}_{n,\varepsilon}(\cdot,0;t)(\xi_1,\ldots,\xi_n)|^2 \mu(d\xi_1) \ldots \mu(d\xi_n)\\
		& =\liminf_{\varepsilon \to 0+}\int_0^{\infty}\e^{-t}\Norm{\widetilde{f}_{n,\varepsilon}(\cdot,0;t)}_{\cH^{\otimes n}}^2 dt \leq 2^{(4-\alpha)n} \frac{1}{(n!)^2} T_n.
	\end{align*}
	This completes the proof of \eqref{E:LapaceNorm} under Assumption \ref{A:Main}.
	\bigskip

	Finally, the case when Assumption \ref{A:White} (resp. Assumption \ref{A:Critical}) holds can be proved similarly
	to Case 1 (resp. Case 2) with $\alpha$ and $\mu(d\xi)$ replaced by $d$ and  $(2\pi)^{-d}d\xi$, respectively.
	This completes the proof of Lemma \ref{L:LpalaceNorm}.
\end{proof}
\bigskip

Now we are ready to prove Theorem \ref{T:iff}.

\begin{proof}[Proof of Theorem \ref{T:iff}.]
	We first assume that Assumption \ref{A:Main} holds with $0<\alpha<d\le 3$.
	By Theorem \ref{T:ExUn}, we need to prove that
	$\sum_{n\geq 0}\theta^n n! \|\widetilde{f}_n(\cdot,0;t)\|_{\cH^{\otimes n}}^2<\infty$, for all $t>0$.
	By the scaling property \eqref{E:scaleTildeFn},
	provided that $(4-\alpha)n+1>0$ (which is true for all $n\ge 1$ since $\alpha<4$),
	\begin{equation}
		\label{E:GammaLapInt}
		\int_0^{\infty}\e^{-t}\Norm{\widetilde{f}_n(\cdot,0;t)}_{\cH^{\otimes n}}^2 dt
		=\Gamma((4-\alpha)n+1)\Norm{\widetilde{f}_n(\cdot,0;1)}_{\cH^{\otimes n}}^2.
	\end{equation}
	By \eqref{E:GammaLapInt}, Lemma \ref{L:LpalaceNorm} and Stirling's formula, there is some constant $c_1>0$ depending only on $\alpha$ (see \eqref{E:StirlingRate} for a justification) such that
	\begin{align*}
		\Norm{\widetilde{f}_n(\cdot,0;1)}_{\cH^{\otimes n}}^2
		\leq \frac{(2^{4-\alpha}C_{\mu}')^n}{\Gamma((4-\alpha)n+1)}
		\leq c_1^n \frac{(2^{4-\alpha}C_{\mu}')^n}{(n!)^{4-\alpha}}.
	\end{align*}
	 Using property \eqref{E:scaleTildeFn} again, we infer that
	\begin{align*}
		\sum_{n\geq 0}\theta^n n! \Norm{\widetilde{f}_n(\cdot,0;t)}_{\cH^{\otimes n}}^2
		\leq \sum_{n\geq 0} \theta^n t^{(4-\alpha)n }c_1^n \frac{(2^{4-\alpha}C_{\mu}')^n}{(n!)^{3-\alpha}}
		< \infty.
	\end{align*}
	Relation \eqref{E:mom-p} follows by noticing that for all $p\ge 2$, by Minkowski's inequality and the hypercontractivity property (see p.62 of Nualart \cite{nualart06}) 
	\begin{align}
		\label{E_:pmoment}
		\Norm{u(t,x)}_p
		& \leq \sum_{n\geq 0} \theta^{n/2} (p-1)^{n/2} (n!)^{1/2} \Norm{ \widetilde{f}_n(\cdot,x;t)}_{\cH^{\otimes n}} \\ \notag
		& \leq \sum_{n\geq 0} \theta^{n/2} (p-1)^{n/2}t^{\frac{4-\alpha}{2}n}c_1^{n/2}\frac{(2^{4-\alpha} C_{\mu}')^{n/2}}{(n!)^{\frac{3-\alpha}{2}}},
	\end{align}
	which is finite for all $t\ge 0$ since $\alpha<d\le 3$.
	\bigskip

	We now study the white noise case (Assumptions \ref{A:White} or \ref{A:Critical}). In this case, the scaling property \eqref{E:scaleTildeFn} still holds with $\alpha$ replaced by $d$, i.e.
	\begin{equation}
		\label{E:scale-white}
		\Norm{\widetilde{f}_n(\cdot,0;t)}_{\cH^{\otimes n}}^2 =t^{(4-d)n}\Norm{\widetilde{f}_n(\cdot,0;1)}_{\cH^{\otimes n}}^2.
	\end{equation}

	Hence, the case of Assumption \ref{A:White} is proved exactly as above with $\alpha$ replaced by $d$.
	Similarly, in the case when Assumption \ref{A:Critical} holds, by the same arguments as above with $\alpha$ replaced by $3$, we see that
	\begin{align*}
		\Norm{\widetilde{f}_n(\cdot,0;1)}_{\cH^{\otimes n}}^2
		\leq \frac{(2\: \left(2\pi\right)^{-3}C_{\text{Leb}}')^n}{\Gamma(n+1)}
		=  \frac{(2\left(2\pi\right)^{-3}C_{\text{Leb}}')^n}{n!},
	\end{align*}
	where $C_{\text{Leb}}'$ is defined in \eqref{E:alpha<4}.
	For any $p\ge 2$, from \eqref{E_:pmoment}, we see that
	\begin{align*}
		\Norm{u(t,x)}_p \le \sum_{n\ge 0} \theta^{n/2} (p-1)^{n/2} t^{n/2} (2\left(2\pi\right)^{-3}C_{\text{Leb}}')^{n/2},
	\end{align*}
	which is convergent provided that $t<T_p' := \frac{4\pi^3}{\theta (p-1) C_{\text{Leb}}'}$.
	Finally, applying (\ref{int}) to the setting $\mu(d\xi)=(2\pi)^{-3}d\xi$ and $\alpha=d=3$
	we have $C_{\text{Leb}}'=\pi^{2}$. With this, the proof of Theorem \ref{T:iff} is complete.
\end{proof}

\section{The case $p=2$}
\label{S:p=2}

The aim of this section is to prove the following result:

\begin{theorem}
	\label{T:p2}
	(i) If Assumption \ref{A:Main} holds with $0<\alpha<d\le 3$, then for all $x\in\R^d$,
	\begin{align}
		\label{E:p2}
		\lim_{t\to \infty}t^{-\frac{4-\alpha}{3-\alpha}}\log \E \left(|u(t,x)|^2\right)
		= \theta^{\frac{1}{3-\alpha}}  \left(\frac{1}{2}\right)^{\frac{\alpha}{2(3-\alpha)}} (3-\alpha) \left(\frac{2\cM^{1/2}}{4-\alpha}\right)^{\frac{4-\alpha}{3-\alpha}}.
	\end{align}
	(ii) Relation \eqref{E:p2} holds also under Assumption \ref{A:White}, with $\alpha$ and $\cM$ replaced by $d$ and $\cM(\delta_0)$, respectively.\\
	(iii) Under Assumption \ref{A:Critical}, let $T_2(\theta)$ be the constant defined in \eqref{E:Tp}.
	Then for any $x\in\R^d$ and $\theta>0$, it holds that
	\begin{align}
		\label{E:p2iii}
		\sum_{n\geq 0} \theta^n n! \Norm{\widetilde{f}_n(\cdot,x;t)}_{\cH^{\otimes n}}^2 \quad
		\begin{cases}
			\quad <\infty & \text{if $t<T_2(\theta)$,} \\
			\quad =\infty & \text{if $t> T_2(\theta)$.}
		\end{cases}
	\end{align}
\end{theorem}

\begin{remark}
	The radius of convergence $T_2$ obtained in part (iii) of Theorem \ref{T:p2} is bounded from below by $T_2'$ obtained in Theorem \ref{T:iff}; see Lemma \ref{L:Tp} below.
	Hence, the solvability condition obtained in part (iii) of Theorem \ref{T:p2} improves part (iii) of Theorem \ref{T:iff}.
\end{remark}

We first prove a lemma.

\begin{lemma}
	\label{L:TwoLim}
	Under Assumption \ref{A:Main} with $0<\alpha<d\le 3$, it holds that
	\begin{align}
		\label{E:lim1}
		&\lim_{n\to \infty} \frac{1}{n} \log \left(\Gamma((4-\alpha)n+1) \Norm{\widetilde{f}_n(\cdot,0;1)}_{\cH^{\otimes n}}^2\right)
		= \log \left[2^{4-\alpha} 2^{-\frac{\alpha}{2}} \cM^{\frac{4-\alpha}{2}}\right]\quad\text{and}\\
		\label{E:lim2}
		&\lim_{n\to \infty} \frac{1}{n} \log \left((n!)^{4-\alpha} \Norm{\widetilde{f}_n(\cdot,0;1)}_{\cH^{\otimes n}}^2\right)
		= \log \left[\left(\frac{2}{4-\alpha}\right)^{4-\alpha} 2^{-\frac{\alpha}{2}} \cM^{\frac{4-\alpha}{2}}\right].
	\end{align}
	Both relations \eqref{E:lim1} and \eqref{E:lim2} hold also under either Assumption \ref{A:White} or \ref{A:Critical},
	with $\alpha$ and $\cM$ replaced by $d$ and $\cM(\delta_0)$, respectively.
\end{lemma}

\begin{proof}
	It suffices to prove \eqref{E:lim1}. Relation \eqref{E:lim2} follows from \eqref{E:lim1}, using the fact that
	\begin{align}
		\label{E:StirlingRate}
		\lim_{n\to \infty}\frac{1}{n}\log \frac{\Gamma(an+1)}{(n!)^a}=a\log a \quad \mbox{for any} \ a>0,
	\end{align}
	which is a consequence of Stirling's formula for the Gamma function.
	The upper bound in \eqref{E:lim1} a direct consequence of Lemma \ref{L:LpalaceNorm} and the scaling property \eqref{E:scaleTildeFn} (resp. \eqref{E:scale-white})
	of $\|\widetilde{f}_n\left(\cdot,x;t\right)\|_{\cH^{\otimes n}}^2$ under Assumption \ref{A:Main} (resp. Assumption \ref{A:White} or \ref{A:Critical}).
	Indeed, under Assumption \ref{A:Main}, by \eqref{E:LapaceNorm} and \eqref{E:GammaLapInt},
	\begin{align*}
		\Gamma((4-\alpha)n+1)\Norm{\widetilde{f}_n(\cdot,0;1)}_{\cH^{\otimes n}}^2
		\leq 2^{(4-\alpha)n} \frac{1}{(n!)^2}T_n,
	\end{align*}
	where $T_n$ is given by \eqref{E:def-Tn}.
	Using \eqref{BCR0} and \eqref{E:rhoM}, we obtain the desired upper bound:
	\begin{align*}
		\limsup_{n\to \infty}\frac{1}{n}\log \left(\Gamma((4-\alpha)n+1)\Norm{\widetilde{f}_n(\cdot,0;1)}_{\cH^{\otimes n}}^2\right)
		\leq \log 2^{4-\alpha}+\log \left(2^{-\frac{\alpha}{2}}\cM^{\frac{4-\alpha}{2}}\right).
	\end{align*}
	Under either Assumption \ref{A:White} or \ref{A:Critical}, one simply replaces $\alpha$ and $\cM$ in the above inequality by $d$ and $\cM(\delta_0)$, respectively.

	\bigskip
	The lower bound in \eqref{E:lim1} is a consequence of the following inequality:
	\begin{equation}
		\label{LB-Tn}
		\frac{1}{(n!)^2}T_n \leq \Gamma\left(\frac{4-\alpha}{2}n+1\right)^2 \Norm{\widetilde{f}_n(\cdot,0;1)}_{\cH^{\otimes n}}^2.
	\end{equation}
	Indeed, another application of Stirling's formula shows that
	\begin{align*}
		\Gamma\left(\frac{(4-\alpha)n}{2}+1\right)^2
		\sim \Gamma\big((4-\alpha)n+1\big) 2^{-(4-\alpha)n} C_n,
	\end{align*}
	where $C_n=2^{-1} (4-\alpha)^{1/2} (2\pi n)^{1/2}$.
	Using \eqref{LB-Tn}, we infer that
	\begin{align*}
		c_{\alpha} \frac{2^{(4-\alpha)n}}{C_n} \frac{1}{(n!)^2}T_n
		\leq \Gamma\big((4-\alpha)n+1\big) \Norm{\widetilde{f}_n(\cdot,0;1)}_{\cH^{\otimes n}}^2,
	\end{align*}
	where $c_{\alpha}>0$ is a constant depending on $\alpha$.
	Using \eqref{BCR0} and the fact that $\frac{1}{n} \log C_n \to 0$ as $n \to \infty$, we obtain the desired lower bound:
	\begin{align*}
		\liminf_{n\to \infty}\frac{1}{n}\log \left(\Gamma((4-\alpha)n+1) \: \Norm{\widetilde{f}_n(\cdot,0;1)}_{\cH^{\otimes n}}^2\right)
		\geq \log 2^{4-\alpha}+\log \rho.
	\end{align*}
	Then plugging the formula \eqref{E:rhoM} for $\rho$ proves \eqref{E:lim1}.
	Therefore, it remains to prove \eqref{LB-Tn}.
	For this, we first assume that Assumption \ref{A:Main} holds. We treat separately the cases $d\leq 2$ and $d=3$ below.
	\bigskip

	{\noindent\em Case 1. Assume that $d\leq 2$.} For any $t>0$ and $\widetilde{t}>0$, define
	\begin{align*}
		J_n(t,\tilde{t})=\int_{(\bR^d)^n}H_n(t,{\bf x})H_n(\tilde{t},{\bf x})d\bf x,
	\end{align*}
	where $H_n(t,x)$ is defined in \eqref{E:Hn}.
	With this notation, relation \eqref{E:keyNorm} becomes:
	\begin{equation}
		\label{fn-Jn}
		\Norm{\widetilde{f}_n(\cdot,0;t)}_{\cH^{\otimes n}}^{2}
		=\frac{1}{(n!)^2}\int_{(\bR^d)^n}H_n^2(t,{\bf x})d{\bf x}
		=\frac{1}{(n!)^2}J_n(t,t).
	\end{equation}
	Hence, $J_n$ satisfies the same scaling property as $\|\widetilde{f}_n(\cdot,0;t)\|_{\cH^{\otimes n}}^{2}$:
	\begin{equation}
		\label{scaling-Jn}
		J_n(t,t)=t^{(4-\alpha)n}J_n(1,1).
	\end{equation}

	Let $\tau$ and $\widetilde{\tau}$ be independent exponential random variables of mean $1$. By Lemma \ref{L:IntHn}, 
	\begin{align*}
		\E\left[J_n(\tau,\widetilde{\tau})\right]
		&=\int_0^{\infty}\int_{0}^{\infty}\e^{-t} \e^{-\tilde t} J_n(t,\tilde{t})dt d\tilde t=\int_{(\bR^d)^n} \left[\int_0^{\infty}\e^{-t} H_n(t,{\bf x})dt \right]^2 d{\bf x} \\
		&=\int_{(\bR^d)^n}|\cF L_n(\xi_1,\ldots,\xi_n)|^2 \mu(d\xi_1)\ldots \mu(d\xi_n)=T_n.
	\end{align*}

	On the other hand, by Cauchy-Schwarz inequality and \eqref{scaling-Jn},
	\begin{align*}
    J_n\left(t,\tilde{t}\:\right) \leq J_n(t,t)^{1/2} J_n\left(\tilde{t},\tilde{t}\:\right)^{1/2}=t^{(4-\alpha)n/2} \tilde{t}^{(4-\alpha)n/2} J_n(1,1)
	\end{align*}
	for any $t>0$ and $\tilde{t}>0$. Hence,
	\begin{align*}
		T_n & = \E\left[J_n(\tau,\widetilde{\tau})\right]
			\leq \E\left[\tau^{(4-\alpha)n/2}\right] \E\left[\widetilde{\tau}^{(4-\alpha)n/2}\right]J_n(1,1)
			=\Gamma\left(\frac{4-\alpha}{2}n+1\right)^2 J_n(1,1) \\
		    & =\Gamma\left(\frac{4-\alpha}{2}n+1\right)^2 (n!)^2 \Norm{\widetilde{f}_n(\cdot,0;1)}_{\cH^{\otimes n}}^2.
	\end{align*}
	\bigskip

	{\noindent\em Case 2. Assume that $d=3$.} Note that the proof below is in fact valid for any $d\geq 3$.
	Instead of the approximation $f_{n,\varepsilon}$ which was used for the proof of \eqref{E:LapaceNorm} in case $d=3$, here we use another approximation, namely
	\begin{align*}
		f_{n,\varepsilon}'(\cdot,x;t)=f_n(\cdot,x;t)* p_{\varepsilon}^{\otimes n},
	\end{align*}
	where $p_{\varepsilon}^{\otimes n}(x_1,\ldots,x_n)=\prod_{k=1}^n p_{\varepsilon}(x_k)$ and
	$p_{\varepsilon}(x)=(2\pi \varepsilon)^{-d/2} \e^{-|x|^2/(2\varepsilon)}$.
	It is known that $f_{n,\varepsilon}'(\cdot,x;t)$ is a $C^{\infty}$-function on $\bR^{nd}$, which belongs to $\cS'(\bR^{nd})$ and has Fourier transform given by
	\begin{align}
		\cF f_{n,\varepsilon}'(\cdot,x;t)(\xi_1,\ldots,\xi_n)
		=\exp\Big(-\frac{\varepsilon}{2}\sum_{k=1}^{n}|\xi_k|^2 \Big) \cF f_{n}(\cdot,x;t)(\xi_1,\ldots,\xi_n).
	\end{align}
	Let $\widetilde{f}_{n,\varepsilon}'(\cdot,x;t)$ be the symmetrization of $f_{n,\varepsilon}'(\cdot,x;t)$.
	Then
	\begin{align*}
		\cF \widetilde{f}_{n,\varepsilon}'(\cdot,x;t)(\xi_1,\ldots,\xi_n)
		=\exp\Big(-\frac{\varepsilon}{2}\sum_{k=1}^{n}|\xi_k|^2 \Big) \cF \widetilde{f}_{n}(\cdot,x;t)(\xi_1,\ldots,\xi_n).
	\end{align*}
	Hence, $|\cF \widetilde{f}_{n,\varepsilon}'(\cdot,x;t)(\xi_1,\ldots,\xi_n)| \leq |\cF \widetilde{f}_{n}(\cdot,x;t)(\xi_1,\ldots,\xi_n)|$ and

	\begin{equation}
		\label{ineq-e}
		\Norm{\widetilde{f}_{n,\varepsilon}'(\cdot,0;t)}_{\cH^{\otimes n}}^2
		\leq \Norm{\widetilde{f}_{n}(\cdot,0;t)}_{\cH^{\otimes n}}^2 \quad \mbox{for any} \ \varepsilon>0.
	\end{equation}
	By \eqref{ineq-e} and Fatou's lemma,
	$\lim_{\varepsilon \to 0+}\|\widetilde{f}_{n,\varepsilon}'(\cdot,0;t)\|_{\cH^{\otimes n}}^2 = \|\widetilde{f}_{n}(\cdot,0;t)\|_{\cH^{\otimes n}}^2$.
	Similar to \eqref{norm-fn-X},
	\begin{align*}
		\Norm{\widetilde{f}_{n,\varepsilon}'(\cdot,0;t)}_{\cH^{\otimes n}}^2
		=\frac{1}{(n!)^2}\int_{(\bR^d)^n} \left[H_{n,\varepsilon}'(t,{\bf x})\right]^2 d{\bf x},
	\end{align*}
	where $H_{n,\varepsilon}'(t,{\bf x})=n! \int_{(\bR^d)^n} \prod_{k=1}^{n}K(x_k-y_k) \widetilde{f}_{n,\varepsilon}'(y_1,\ldots,y_n,0;t)d{\bf y}$.
	Similar to Lemma \ref{L:IntHn},
	\begin{align*}
		\int_{(\bR^d)^n} \left[ \int_0^{\infty} \e^{-t}H_{n,\varepsilon}'(t,{\bf x}) dt\right]^2 d{\bf x}
		= \int_{(\bR^d)^n} |\cF L_{n,\varepsilon}'(\xi_1,\ldots,\xi_n)|^2 \mu(d\xi_1) \ldots \mu(d\xi_n),
	\end{align*}
	where
	$L_{n,\varepsilon}'(y_1,\ldots,y_n)=n!\int_0^{\infty}e^{-t} \widetilde{f}_{n,\varepsilon}'(y_1,\ldots,y_n,0;t)dt$ and
	\begin{align*}
		\cF L_{n,\varepsilon}'(\xi_1,\ldots,\xi_n)
		= \exp \Big(-\frac{\varepsilon}{2}\sum_{k=1}^{n}|\xi_k|^2 \Big) \sum_{\sigma \in \Sigma_n}
		  \prod_{k=1}^{n} \frac{1}{1+|\sum_{j=k}^n \xi_{\sigma(j)}|^2}.
	\end{align*}

	Denote
	\begin{align*}
		J_{n,\varepsilon}'(t,\tilde{t})
		=\int_{(\bR^d)^n} H_{n,\varepsilon}'(t,{\bf x})H_{n,\varepsilon}'(\tilde{t},{\bf x})d{\bf x}.
	\end{align*}
	With this notation,
	$\|\widetilde{f}_{n,\varepsilon}'(\cdot,0;t)\|_{\cH^{\otimes n}}^2 =\frac{1}{(n!)^2} J_{n,\varepsilon}'(t,t)$. By
	\eqref{ineq-e}, it follows that
	\begin{equation}
		\label{ineq-e2}
		J_{n,\varepsilon}'(t,t) \leq J_n(t,t) \quad \mbox{for any} \ \varepsilon>0,
	\end{equation}
	where $J_n(t,t)=(n!)^2 \|\widetilde{f}_n(\cdot,0;t)\|_{\cH^{\otimes n}}^2$.
	Let $\tau$ and $\widetilde{\tau}$ be independent exponential random variables of mean $1$.
	Then
	\begin{align*}
		\E[J_{n,\varepsilon}'(\tau,\widetilde{\tau})]&=\int_{(\bR^d)^n} \left[ \int_0^{\infty} \e^{-t}H_{n,\varepsilon}'(t,{\bf x}) dt\right]^2 d{\bf x}\\
		&=\int_{(\bR^d)^n} \exp \Big(-\varepsilon \sum_{k=1}^{n}|\xi_k|^2 \Big)
		\left[\sum_{\sigma \in \Sigma_n} \prod_{k=1}^{n} \frac{1}{1+|\sum_{j=k}^n \xi_{\sigma(j)}|^2}\right]^2\mu(d\xi_1) \ldots \mu(d\xi_n)\\
		&=:T_{n,\varepsilon}'.
	\end{align*}
	By the Cauchy-Schwarz inequality, \eqref{ineq-e2} and the scaling property of $J_n$,
	\begin{align*}
		J_{n,\varepsilon}'(t,\tilde{t})
		\leq J_{n,\varepsilon}'(t,t)^{1/2} J_{n,\varepsilon}'(\tilde{t},\tilde{t})^{1/2}
		\leq J_{n}(t,t)^{1/2}J_{n}(\tilde{t},\tilde{t})^{1/2}
		=t^{\frac{4-\alpha}{2}n} \tilde{t}^{\frac{4-\alpha}{2}n}J_{n}(1,1).
	\end{align*}
	Hence, $T_{n,\varepsilon}' =\E[J_{n,\varepsilon}'(\tau,\widetilde{\tau})]$, which is bounded from above by
	\begin{align*}
		\Gamma\left(\frac{4-\alpha}{2}n+1\right)^2J_{n}(1,1)
		=\Gamma\left(\frac{4-\alpha}{2}n+1\right)^2(n!)^2 \Norm{\widetilde{f}_n(\cdot,0;1)}_{\cH^{\otimes n}}^2.
	\end{align*}
	Relation \eqref{LB-Tn} follows since $\lim_{\varepsilon \to 0+}T_{n,\varepsilon}'=T_n$ by the dominated convergence theorem.

	\bigskip
	Under Assumption \ref{A:White} (resp. \ref{A:Critical}), one can carry out the same arguments as those in Case 1 (resp. Case 2) with $\alpha$ replaced by $d$ and
	$\mu(d\xi)$ by $\left(2\pi\right)^{-d}d\xi$.
\end{proof}
\bigskip

\begin{proof}[Proof of Theorem \ref{T:p2}]
	Provided that there is a $L^2(\Omega)$ solution $u(t,x)$, thanks to the scaling property \eqref{E:scaleTildeFn}, we should have
	\begin{align}
		\notag
		\E\left(|u(t,x)|^2\right)
		&=\sum_{n\geq 0}\theta^n n! \,\Norm{\widetilde{f}_n(\cdot,0;t)}_{\cH^{\otimes n}}^{2}
		=\sum_{n\geq 0}\theta^n n! \,t^{(4-\alpha)n} \Norm{\widetilde{f}_n(\cdot,0;1)}_{\cH^{\otimes n}}^{2}\\
		&=\sum_{n\geq 0}\theta^n \frac{t^{(4-\alpha)n}}{(n!)^{3-\alpha}} R_n\qquad
		\text{with $R_n:= (n!)^{4-\alpha} \Norm{\widetilde{f}_n(\cdot,0;1)}_{\cH^{\otimes n}}^{2} $}.
		\label{E_:p_2}
	\end{align}
  For (i), part (i) of Theorem \ref{T:iff} guarantees the existence of an $L^2(\Omega)$-solution for all $t>0$ and $x\in \bR^d$.
	We apply Lemma \ref{Rn-compare} with $x_n=\theta^n \frac{1}{(n!)^{3-\alpha}}$ and the above $R_n$.
	By \eqref{E:lim2}, we know that
	\begin{align}
		\label{E:R}
		\lim_{n \to \infty} \frac{1}{n}\log R_n
		= \log R, \quad \mbox{with} \ R
		:= \left(\frac{2}{4-\alpha}\right)^{4-\alpha} 2^{-\frac{\alpha}{2}} \cM^{\frac{4-\alpha}{2}}.
	\end{align}
	So, for deriving the asymptotic behaviour of $t^{-\beta}\log \E\left(|u(t,x)|^2\right)$ (for suitable $\beta>0$),
	we can replace $R_n$ by $R^n$. More precisely, if we can find $\beta>0$ such that
	\begin{equation}
		\label{E4:A-series}
		\lim_{t \to \infty}\frac{1}{t^{\beta}} \log \sum_{n \geq 0} \theta^n \frac{1}{(n!)^{3-\alpha}} R^n t^{(4-\alpha)n}
		= A
	\end{equation}
	then
	\begin{align*}
		\lim_{t \to \infty}\frac{1}{t^{\beta}} \log \E\left(|u(t,x)|^2\right)
		=\lim_{t \to \infty}\frac{1}{t^{\beta}} \log \sum_{n \geq 0} \theta^n \frac{1}{(n!)^{3-\alpha}} R_n t^{(4-\alpha)n}
		= A.
	\end{align*}
	To find the suitable value $\beta$, we use Lemma \ref{L:Gamma}. More precisely,
	\begin{align*}
		\frac{1}{(\theta R)^{1/(3-\alpha)}} \lim_{t \to \infty} \frac{1}{t^{(4-\alpha)/(3-\alpha)}} \log \sum_{n\geq 0}
		\frac{\theta^n R^n t^{(4-\alpha)n}}{(n!)^{3-\alpha}}
		& = \lim_{t \to \infty} \frac{1}{(\theta Rt^{4-\alpha})^{1/(3-\alpha)}} \log \sum_{n \geq 0} \frac{(\theta R t^{4-\alpha})^n}{(n!)^{3-\alpha}}\\
		& = \lim_{t \to \infty} \frac{1}{t^{1/(3-\alpha)}} \log \sum_{n \geq 0} \frac{t^n}{(n!)^{3-\alpha}}
		  = 3-\alpha.
	\end{align*}
	This shows that if we choose $\beta=\frac{4-\alpha}{3-\alpha}$ (see \eqref{E:beta}),
	then relation \eqref{E4:A-series} holds with $A=(3-\alpha)(\theta R)^{1/(3-\alpha)}$. Summarizing, we obtain:
	\begin{align*}
		\lim_{t \to \infty}\frac{1}{t^{\beta}} \log \E\left(|u(t,x)|^2\right) = (3-\alpha)(\theta R)^{1/(3-\alpha)}.
	\end{align*}
	The conclusion of part (i) follows using the definition of $R$ in \eqref{E:R}.
	Part (ii) can be proved in the same way as above with $\alpha$ and $\cM$ replaced by $d$ and $\cM(\delta_0)$, respectively.
	\bigskip

As for (iii), by the scaling property \eqref{E:scaleTildeFn},
	\begin{align*}
	\sum_{n\geq 0} \theta^n n! \Norm{\widetilde{f}_n(\cdot,x;t)}_{\cH^{\otimes n}}^2
		= \sum_{n\ge 0} \theta^n t^n R_n \quad \text{with $R_n=n! \Norm{\widetilde{f}_n(\cdot,0;1)}_{\cH^{\otimes n}}^{2}$.}
	\end{align*}
	By Lemma \ref{L:TwoLim}, $\lim_{n \to \infty} \frac{1}{n}\log R_n = \log R$ with $R = 2^{-1/2} \cM(\delta_0)^{1/2}$. By the Cauchy-Hadamard theorem, it follows that
	\begin{align*}
		\sum_{n\geq 0} \theta^n t^n R_n \quad
		\begin{cases}
			\quad <\infty & \text{if $\theta t<R^{-1}$,} \\
			\quad =\infty & \text{if $\theta t> R^{-1}$.}
		\end{cases}
	\end{align*}
	This proves part (iii) of Theorem \ref{T:p2}.
\end{proof}

\section{Upper bound}
\label{S:upper}
In this section, we use the $L^2(\Omega)$ asymptotics obtained in Theorem \ref{T:p2} and the hypercontractivity property given by Theorem \ref{T:Khoa} to prove the following theorem:
\begin{theorem}
	\label{T:Upper}
	(i) Under Assumption \ref{A:Main} with $0<\alpha<d\le 3$,
	it holds that
	\begin{align}
		\label{E:Upper}
		\limsup_{t_p\to \infty}t_p^{-\frac{4-\alpha}{3-\alpha}}\log \Norm{u(t,x)}_p
		&\le \theta^{\frac{1}{3-\alpha}}  \left(\frac{1}{2}\right)^{\frac{\alpha}{2(3-\alpha)}} \frac{3-\alpha}{2} \left(\frac{2\cM^{1/2}}{4-\alpha}\right)^{\frac{4-\alpha}{3-\alpha}}.
	\end{align}
	where $t_p$ and $\cM$ are defined in \eqref{E:tp} and \eqref{E:cM}, respectively. \\
	(ii) Inequality \eqref{E:Upper} holds under Assumption \ref{A:White} with $\alpha$ and $\cM$ replaced by $d$ and $\cM(\delta_0)$, respectively.\\
	(iii) Under Assumption \ref{A:Critical}, for any $p\ge 2$, equation \eqref{E:SWE} has a solution in $L^p(\Omega)$ for all $t\in(0,T_p)$ and $x\in\R^d$,
	where $T_p$ is defined in \eqref{E:Tp}.
	Moreover, $T_p \ge T_p'$, where $T_p'$ is defined in \eqref{E:Tp'}.
\end{theorem}

\begin{proof}
	(i) For any $\lambda>0$, we denote by $u_{\lambda}$ the solution of equation \eqref{E:SWE} with $\theta$ replaced by $\lambda$.
	Denote $u=u_{\theta}$.
	By L\^e's moment comparison theorem given by Theorem \ref{T:Khoa},
	\begin{align*}
		\Norm{u(t,x)}_{p} \leq \Norm{u_{(p-1)\theta}(t,x)}_2,\quad\text{for all $p\ge 2$.}
	\end{align*}
	By the scaling property \eqref{E:scaleTildeFn},
	\begin{align*}
		\Norm{u_{(p-1)\theta}(t,x)}_{2}^2
		& = \sum_{n\geq 0} n! \, \theta^n (p-1)^n \Norm{\widetilde{f}_n(\cdot,x;t)}_{\cH^{\otimes n}}^2
		= \sum_{n\geq 0} n! \, \theta^n \Norm{\widetilde{f}_n\left(\cdot,x;(p-1)^{\frac{1}{4-\alpha}}t\right)}_{\cH^{\otimes n}}^2\\
		& = \Norm{u\left((p-1)^{1/(4-\alpha)}t,x\right)}_2^2.
	\end{align*}
	This leads to the following important moment inequality:
	\begin{equation}
		\label{E:p-2}
		\Norm{u(t,x)}_{p} \leq \Norm{u\left( (p-1)^{1/(4-\alpha)}t,x\right)}_2 =  \Norm{u\left( t_p,x\right)}_2,
	\end{equation}
	which implies that
	\begin{align*}
		\limsup_{t_p\to \infty}t_p^{-\frac{4-\alpha}{3-\alpha}}\log \Norm{u(t,x)}_p
		&\leq \limsup_{t_p\to \infty}t_p^{-\frac{4-\alpha}{3-\alpha}}\log  \Norm{u\left(t_p,0\right)}_2.
	\end{align*}
	Then an application of Theorem \ref{T:p2} proves \eqref{E:Upper}.

	Part (ii) can be proved in the same way as above with $\alpha$ and $\cM$ replaced by $d$ and $\cM(\delta_0)$. We will not repeat this here.

	To prove part (iii), we note that relation \eqref{E:p-2} continues to hold under Assumption C, with $\alpha=d=3$ and $t_p=(p-1)t$. The right-hand side of \eqref{E:p-2} is equal to the square root of the series in \eqref{E:p2iii} with $t$ replaced by $t_p$. We note that $t_p< T_2(\theta)$ is equivalent to $t<T_p(\theta)$, recalling the definition \eqref{E:Tp} of $T_p(\theta)$.
	Finally, the property $T_p\ge T_p'$ is proved in Lemma \ref{L:Tp} below.
\end{proof}

\begin{lemma}
	\label{L:Tp}
	When $d=3$, it holds that $\cM(\delta_0)\le 1/(2\pi^4)$. As a consequence,
	for any $\theta>0$ and $p>1$,
	\begin{align}
		\label{E:TpTp'}
		T_p
		:= \frac{\sqrt{2}}{\theta (p-1)\sqrt{\cM(\delta_0)}}
		\ge \frac{2\pi^2}{\theta(p-1)}
		\ge \frac{4\pi}{\theta(p-1)}
		= :T_p'.
	\end{align}
\end{lemma}
\begin{proof}
	Let us find the optimal constant $C>0$ such that $\cM(\delta_0) \le C$,
	or equivalently,
	\begin{align*}
		\left(\int_{\R^3}g^4(x)dx\right)^{1/2}
		\le C +\frac{1}{2}\int_{\R^3} \left|\nabla g(x)\right|^2 dx, \quad\text{for all $g\in\cF_d$.}
	\end{align*}
	Fix an arbitrary $g\in \cF_d$. By H\"older inequality,
	\begin{align*}
		\left(\int_{\R^3}g^4(x)dx\right)^{1/2}
		\le \left(\int_{\R^3}g(x)^2dx \right)^{1/4} \left(\int_{\R^3}g(x)^6dx \right)^{1/4}
		=   \Norm{g}_{L^6(\R^3)}^{3/2}.
	\end{align*}
	By the Sobolev inequality with the optimal constant (see, e.g., \cite{DPD02} for references  therein),
	\begin{align}
		\label{E:SobOptA}
		\Norm{g}_{L^6(\R^3)} \le A \Norm{\nabla g}_{L^2(\R^3)}, \quad \text{with $A=3^{-1/2}\left(\frac{2}{\pi}\right)^{2/3}$}.
	\end{align}
	By setting $y = \Norm{\nabla g}_{L^2(\R^3)}^{1/2}$, the problem reduces to find the constant $C$ such that
	\begin{align}
		\label{E:varphiy}
		\varphi(y):= A^{3/2} y^3 - \frac{1}{2} y^4 \le C \quad\text{for all $y\ge 0$.}
	\end{align}
	Elementary computations show that the critical points of $\varphi$ are $y_1=y_2=0$ and $y_3=(3/2) A^{3/2}$.
	Hence, inequality \eqref{E:varphiy} is equivalent to $\varphi(0)\ge 0$ and $\varphi\left(y_3\right)\ge 0$, from which one finds the value of $C$, namely,
	$C=\varphi\left(y_3\right)=(27A^6)/32=1/(2\pi^4)$.
This proves the first inequality in \eqref{E:TpTp'}.
	The second inequality in \eqref{E:TpTp'} clear.
\end{proof}
\section{Lower bound}
\label{S:lower}

The goal of this section is to prove the following theorem:

\begin{theorem}
	\label{T:Lowbound}
	Under Assumption \ref{A:Main} with $0<\alpha<d\le 3$, it holds that
	\begin{equation}
		\label{E:Low}
		\liminf_{t_p\to \infty} t_p^{-\frac{4-\alpha}{3-\alpha}} \log \Norm{u(t,x)}_p
		\geq
		\theta^{\frac{1}{3-\alpha}}
		\left(\frac{1}{2}\right)^{\frac{\alpha}{2(3-\alpha)}}\frac{3-\alpha}{2}
		\left(\frac{2\cM^{1/2}}{4-\alpha}\right)^{\frac{4-\alpha}{3-\alpha}},
	\end{equation}
	where $t_p$ and $\cM$ are defined in \eqref{E:tp} and \eqref{E:cM}, respectively.
	Relation \eqref{E:Low} also holds under Assumption \ref{A:White}, with $\alpha$ and $\cM$ replaced by $d$ and $\cM(\delta_0)$, respectively.
\end{theorem}

\bigskip

This theorem will be proved at the end of this section.
We need first introduce some notation and prove some auxiliary results.
For any $\phi \in L_{\bC}^2(\mu)$, let

\begin{align}
	\label{E:Wnt}
	W_n(t,\phi):=\int_{[0,t]_<^n} \int_{(\bR^d)^n} \prod_{k=1}^n \phi(\xi_k) \prod_{k=1}^{n} \cF G(s_k-s_{k-1},\cdot)(\xi_k+\ldots+\xi_n) \mu(d\xi_1) \ldots \mu(d\xi_n) d{\bf s},
\end{align}
with $s_0=0$. We first give conditions under which $W_n(t,\phi)$ is well defined:

\begin{lemma}
	\label{L:LaplaceW}
	If the measure $\mu$ satisfies the first relation in \eqref{E:alpha<4}, then  $W_n(t,\phi)$
	is well defined for any $d\ge 1$, $t>0$, and $\phi\in L_{\bC}^2(\mu)$. Moreover,
	\begin{align}
		\label{E:LaplaceW}
		\int_0^{\infty}\e^{-t}W_n(t,\phi)dt
		= \int_{(\bR^d)^n} \prod_{k=1}^{n}\phi(\xi_k) \prod_{k=1}^{n}
		\frac{1}{1+|\xi_{k}+\ldots+\xi_{n}|^2} \mu(d\xi_1)\ldots \mu(d\xi_n).
	\end{align}
\end{lemma}

\begin{proof}
	By Fubini's theorem, the left-hand side of \eqref{E:LaplaceW} is equal to
	\begin{align*}
		\int_{(\bR^{d})^n} \prod_{k=1}^{n}\phi(\xi_k)  \int_0^{\infty}\e^{-t}
		\int_{[0,t]_{<}^{n}} \prod_{k=1}^{n} \cF
		G(s_k-s_{k-1},\cdot)\left(\xi_k+\ldots+\xi_n\right) d{\bf s} dt
		\mu(d\xi_1) \ldots \mu(d\xi_n),
	\end{align*}
	which is equal to, by the change of variables $t-s_n=u$ and $s_k-s_{k-1}=u_k$ for $k=1,\ldots,n$,
	\begin{align*}
		\int_{(\bR^{d})^n} \prod_{k=1}^{n}\phi(\xi_k) \left( \int_{0}^{\infty} \e^{- u}du\right)
		\prod_{k=1}^{n}\left( \int_0^{\infty}\e^{-u_k} \cF G(u_k,\cdot)\left(\xi_k+\ldots+\xi_n\right)du_k\right)
		\mu(d\xi_1) \ldots \mu(d\xi_n).
	\end{align*}
 	Noticing that
	$\int_0^{\infty}\e^{-t} \cF G(t,\cdot)(\xi)dt=\int_0^{\infty}\e^{-t}
	\frac{\sin(t|\xi|)}{|\xi|}dt=1/(1+|\xi|^2)$ for all $\xi\in \mathbb{C}^d$, the above quantity is equal to
	\begin{align*}
		\int_{(\bR^d)^n} \prod_{k=1}^{n}\phi(\xi_k) \prod_{k=1}^{n}
		\frac{1}{1+|\xi_k+\ldots+\xi_n|^2} \mu(d\xi_1)\ldots \mu(d\xi_n).
	\end{align*}
	Finally, one can apply the Cauchy-Schwarz inequality and then use the same arguments as those leading to \eqref{E_:alpha<4}
	to see that the above quantity is finite.
\end{proof}

Our interest in the quantity $W_n(t,\phi)$ stems from the following proposition.

\begin{proposition}
	\label{LB-1}
	For any $f\in \cH$, $t>0$, and $p,q>1$ with $1/p+1/q=1$, it holds that
	\begin{align*}
		\Norm{u(t,0)}_p \geq
		\exp\left\{-\frac{1}{2}(q-1)\|f\|_{\cH}^2 \right\}
		\left|\sum_{n\geq 0}\theta^{n/2} W_n\left(t,\cF f\right)\right|.
	\end{align*}
	As a consequence, the series $\left|\sum_{n\geq 0} \theta^{n/2} W_n\left(t,\cF f\right)\right|$
	converges provided that $\Norm{u(t,0)}_p<\infty$, which is the case if Assumption \ref{A:Main} holds with
	$0<\alpha<d\leq 3$ or Assumption \ref{A:White} holds (see Theorem \ref{T:iff}).
\end{proposition}

\begin{proof}
	It suffices to prove the inequality.
	For any distribution $f \in \cH$, let
	\begin{align*}
		Z_f=\exp\left(W(f)-\frac{1}{2}\Norm{f}_{\cH}^2\right)=\sum_{n\geq 0}\frac{1}{n!}I_n(f^{\otimes n}).
	\end{align*}
	Here $f^{\otimes n}$ is the distribution in $\cS'(\bR^{nd})$ whose Fourier transform is given by:
	$\cF f^{\otimes n}(\xi_1,\ldots,\xi_n)=\prod_{k=1}^{n}\cF f(\xi_k)$.
	If $f$ is a function, then $f^{\otimes n}$ is the function $f^{\otimes n}(x_1,\ldots,x_n)=\prod_{k=1}^{n}f(x_k)$.

	Using H\"older's inequality and the fact that $W(f)\sim N(0,\Norm{f}_{\cH}^2)$, we obtain:
	\begin{equation}
		\label{E-uZ-Holder}
		\left|\bE [u(t,0)Z_f]\right| \leq \Norm{u(t,0)}_p \Norm{Z_f}_{q}
		=\Norm{u(t,0)}_p \exp\left\{\frac{1}{2}(q-1)\Norm{f}_{\cH}^2\right\}.
	\end{equation}
	By the orthogonality of the Wiener chaos spaces,
	\begin{equation}
		\label{E-uZ}
		\bE\left[u(t,0)Z_f\right] = \sum_{n\geq0} \theta^{n/2} \frac{1}{n!}\bE[I_n(f_n(\cdot,0;t)) I_n(f^{\otimes n})]
		                          = \sum_{n\geq0} \theta^{n/2}\langle \widetilde{f}_n(\cdot,0;t), f^{\otimes n}\rangle_{\cH^{\otimes n}}.
	\end{equation}
	Using the definition of $\widetilde{f}_n(\cdot,0;t)$, we see that
	\begin{align}
		\notag
		   \langle \widetilde{f}_n(\cdot,0;t), f^{\otimes n}\rangle_{\cH^{\otimes n}}
		=& \int_{(\bR^d)^n} \prod_{k=1}^{n}\cF f(\xi_k) \overline{\cF \widetilde{f}_n(\cdot,x;t)(\xi_1,\ldots,\xi_n)} \mu(d\xi_1) \ldots \mu(d\xi_n) \\ \notag
		=& \frac{1}{n!} \sum_{\sigma\in\Sigma_n} \int_{[0,t]_{<}^n} d{\bf t} \int_{(\bR^d)^n} \mu(d\xi_1) \ldots \mu(d\xi_n) \prod_{k=1}^{n}\cF f(\xi_k) \\ \notag
		 &\times  \prod_{k=1}^{n} \cF G(t_{k+1}-t_k,\cdot)\left(\xi_{\sigma(1)}+\ldots+\xi_{\sigma(k)}\right) \\ \notag
		=& \int_{[0,t]_{<}^n} d{\bf t} \int_{(\bR^d)^n} \mu(d\xi_1) \ldots \mu(d\xi_n) \prod_{k=1}^{n}\cF f(\xi_k) \\ \notag
		 &\times  \prod_{k=1}^{n} \cF G(t_{k+1}-t_k,\cdot)\left(\xi_1+\ldots+\xi_k\right) \\
		=& W_n(t,\cF f),
		\label{product-f}
	\end{align}
	where $t_{n+1}=t$ and the last step is due to the change of variables $s_1=t-t_n,\ldots,s_n=t-t_1$ and $\xi_1'=\xi_n,\ldots,\xi_n'=\xi_1$.
	The conclusion follows from \eqref{E-uZ-Holder}, \eqref{E-uZ} and \eqref{product-f}.
\end{proof}

If Assumption \ref{A:Main} holds,
for any $f \in \cH$ and $t>0$, let $f_t$ be {\it the time-scaled distribution}, which is a distribution in $\cS'(\bR^d)$ whose Fourier transform is
\begin{align}
	\label{E:scale_f}
	\cF f_t(\xi)
	=t^{(2-\alpha)(\beta-1)}\cF f(t^{1-\beta}\xi)
  =t^{\frac{2-\alpha}{3-\alpha}}\cF f\left(t^{1/(\alpha-3)} \xi\right)
	\quad \mbox{for all} \ \xi \in \mathbb{C}^d.
\end{align}
where $\beta$ is given by \eqref{E:beta}.
Note that $f_t\in \cH$.
If $f$ is a function, then $f_t$ is also a function with the scaling property:
\begin{align*}
	f_t(x)=t^{(2-\alpha+d)(\beta-1)}f(t^{\beta-1}x)
	= t^{\frac{2-\alpha+d}{3-\alpha}} f\left(t^{1/(3-\alpha)}x\right),
	\quad \mbox{for all} \ x \in \bR^d.
\end{align*}
If Assumption \ref{A:White} holds, $f_t$ is defined similarly with $\alpha$ replaced by $d$ and $\beta$ given by \eqref{E:beta}.
Regarding $f_t$, we have the following scaling result.

\begin{lemma} \label{scaling-ft}
	Suppose that parts (i) and (ii) of Assumption \ref{A:Main} hold with $\alpha<4$ or Assumption \ref{A:White} holds. Then for any $f \in \cH$ and $\tau>0$,
	\begin{align*}
		\Norm{f_\tau}_{\cH}^2=\tau^{\beta}\|f\|_{\cH}^2 \quad \mbox{and} \quad
		W_n\left(t,\cF f_\tau\right)=W_n\left( t\:\tau^{\beta-1} ,\cF f\right).
	\end{align*}
\end{lemma}

\begin{proof} We treat only the case of Assumption \ref{A:Main}. The white noise case can be proved in a similar way.
	Using the change of variable $\xi'=\tau^{1-\beta}\xi$ and the scaling property of $\mu$, we see that:
	\begin{align*}
		\Norm{f_\tau}_{\cH}^2
		=\tau^{2(2-\alpha)(\beta-1)} \int_{\bR^d}|\cF f(\tau^{1-\beta} \xi)|^2 \mu(d\xi)
		=\tau^{(\beta-1)(4-\alpha)}\Norm{f}_{\cH}^2.
	\end{align*}
	The desired relation follows since $(\beta-1)(4-\alpha)=\beta$.
	The second relation is proved by the change of variable $\xi_k'=\tau^{1-\beta}\xi_k$, applying the scaling properties of
	$\mu$ and $\cF G$ in \eqref{E:scaleMu} and  \eqref{E:scaleFG}, respectively,
	and then anther change of variable $s_k'=\tau^{\beta-1}s_k$ as follows:
	\begin{align*}
		W_n\left(t,\cF f_\tau\right)
		&=\tau^{n(2-\alpha)(\beta-1)}\int_{[0,t]_<^n} d\mathbf{s} \int_{(\bR^d)^n}  \mu(d\xi_1)\ldots \mu(d\xi_n) \\
		&\quad \times \prod_{k=1}^n \cF f\left(\tau^{1-\beta}\xi_k\right) \prod_{k=1}^{n}\cF G(s_k-s_{k-1},\cdot)\left(\xi_k+\cdots+ \xi_n\right)\\
		&=\tau^{2n(\beta-1)} \int_{[0,t]_<^n} d\mathbf{s} \int_{(\bR^d)^n}  \mu(d\xi_1)\ldots \mu(d\xi_n) \\
		&\quad \times  \prod_{k=1}^n\cF f(\xi_k) \prod_{k=1}^{n}\cF G(s_k-s_{k-1},\cdot)\left(\tau^{\beta-1}\left(\xi_k+\cdots + \xi_n\right)\right)\\
		&=\tau^{n(\beta-1)} \int_{[0,t]_<^n}  d\mathbf{s} \int_{(\bR^d)^n} \mu(d\xi_1)\ldots \mu(d\xi_n) \\
		&\quad \times  \prod_{k=1}^n \cF f(\xi_k)
		\prod_{k=1}^{n}\cF G\left(\tau^{\beta-1}(s_k-s_{k-1}),\cdot\right)\left(\xi_k+\cdots+\xi_n\right)\\
		&=W_n\left(t\tau^{\beta-1},\cF f\right).
	\end{align*}
	This proves the lemma.
\end{proof}

\begin{proposition}
	\label{P:LB}
	If Assumption \ref{A:Main} holds with $0<\alpha<d\leq 3$ or Assumption \ref{A:White} holds, then for any $f \in \cH$, $t>0$, and $p,q>1$ with $1/p+1/q=1$, it holds that
	\begin{equation}
		\label{LB-2}
		\Norm{u(t,0)}_p \geq
		\exp\left\{-\frac{1}{2}\: t_p^{\beta}\: \Norm{f}_{\cH}^2 \right\}
		\left|\sum_{n\geq 0} {\theta^{n/2}} W_n\left(t_p^{\beta}, \cF f\right)\right|,
	\end{equation}
	where $t_p$ is defined in \eqref{E:tp}.
\end{proposition}
\begin{proof}
	For any $f \in \cH$, an application of  Proposition \ref{LB-1} to $f_{\tau}$ with $\tau = (p-1)t$,
	followed by Lemma \ref{scaling-ft}, shows that
	\begin{align*}
		\Norm{u(t,0)}_p \geq
		\exp\left\{-\frac{1}{2}\:(q-1) (p-1)^\beta t^{\beta}\: \Norm{f}_{\cH}^2 \right\}
		\left|\sum_{n\geq 0} {\theta^{n/2}} W_n\left(t \left([p-1]t\right)^{\beta-1}, \cF f\right)\right|.
	\end{align*}
	This proves \eqref{LB-2} because $(q-1)(p-1)^{\beta}t^\beta = (p-1)^{\beta-1} t^\beta = (p-1)^{\frac{1}{3-\alpha}}t^\beta=t_p^\beta$.
\end{proof}

Note that $W_n(t,\phi)$ may be negative. Suppose that Assumption \ref{A:Main} holds.
In order to get rid of the absolute value in \eqref{LB-2}, we need to identify $f\in\cH$ for which $W_n(t,\cF f)$ is nonnegative.
For this purpose, we will need to both restrict our dimensions to $d\le 3$ and restrict $\cH$ to the following space
\begin{align}
  \mathcal{H}_+:= \left\{ f\in \mathcal{H}:\: \text{$f$ is
  a nonnegative and nonnegative definite function}\right\}.
\end{align}
With these restrictions, by the Plancherel theorem, we see that
\begin{align}\begin{aligned}
	\label{E:Untf}
	W_n(t,\cF f)
	= & 
\int_{[0,t]_<^n} \int_{\left(\bR^d\right)^n} \prod_{k=1}^{n} (f*\gamma)(x_k)
	    \prod_{k=1}^{n}G\left(s_k-s_{k-1},x_{k}-x_{k-1}\right) d{\bf x}d{\bf s}\\
	=:& U_n(t,f), \qquad \text{for all $f\in\cH_+$ and $d\le 3$,}
\end{aligned}\end{align}
with $s_0=0$ and $x_0=0$.
Since $\prod_{k=1}^n G\left(s_k-s_{k-1},x_k-x_{k-1}\right)$ with $(s_1,\cdots,s_n)$ fixed
is a nonnegative measure with compact support on $\left(\R^{d}\right)^n$, which is true only for $d\le 3$,
and since the other product $(x_1,\cdots,x_n) \mapsto \prod_{k=1}^{n} \left(f*\gamma\right)(x_k)$
is a nonnegative function on $\left(\R^{d}\right)^n$, the $d \mathbf{x}$-integral in \eqref{E:Untf} is
nonnegative, i.e,
\begin{align}
	\label{E:W=U>0}
	0\le W_n(t,\cF f) = U_n(t,f) <\infty,\qquad \text{for all $f\in\cH_+$, $n\ge 1$, $t>0$ and $d\le 3$,}
\end{align}
where the finiteness is a consequence of Proposition \ref{P:LB}.
Relations \eqref{E:Untf} and \eqref{E:W=U>0} continue to hold under Assumption \ref{A:White}, in which case
\begin{align*}
	U_n(t,f):= \int_{[0,t]_{<}^n}\int_{(\bR^d)^n}\prod_{k=1}^{n}f(x_k)
		   \prod_{k=1}^{n}G\left(s_k-s_{k-1},x_k-x_{k-1}\right)d{\bf x}d{\bf s}.
\end{align*}

Now let us define
\begin{align}
	\label{E:Wnt2}
	W_n(t)=\sup_{f\in\cH;\: \Norm{f}_\cH =1} W_n(t,\cF f)\quad\text{and}\quad
	U_n(t)=\sup_{ f \in \cH_+;\: \|f\|_{\cH}=1}U_n(t,f).
\end{align}

Then, when $d\le 3$, we have that
\begin{align*}
	W_n(t) \geq \sup_{f \in \cH_+;\|f\|_{\cH}=1}W_n(t,\cF f)  = U_n(t) \ge 0 .
\end{align*}

\begin{theorem}
	\label{liminf-EW}
	Suppose that Assumption \ref{A:Main} holds with $0<\alpha<d\le 3$. If $\tau$ is an exponential random variable of mean $1$, then
	\begin{align} \label{E:EU}
		\liminf_{n\to \infty}\frac{1}{n} \log \bE \left[U_n(\tau)\right] \geq \log \left(2^{-\frac{\alpha}{4}} \cM^{\frac{4-\alpha}{4}}\right).
	\end{align}
	Relation \eqref{E:EU} also holds under Assumption \ref{A:White} with $\alpha$ and $\cM$ replaced by $d$ and $\cM(\delta_0)$, respectively.
\end{theorem}

\begin{proof}
	We treat only the case when Assumption \ref{A:Main} holds. The white noise case can be proved in a similar way.
	The proof follows using arguments similar to those in Section 3 of Bass et al \cite{BCR09}.
	Notice that
	\begin{align}
		\label{E_:EU}
		\bE \left[U_n(\tau)\right]
		=\int_0^{\infty}\e^{-t}U_n(t)dt
		\geq \sup_{f\in\cH_+;\: \Norm{f}_\cH=1} \int_0^{\infty}\e^{-t}U_n(t,f)dt.
	\end{align}
	Fix an arbitrary $f\in\cH_+$ with $\Norm{f}_\cH=1$ and set $\phi=\cF f$.
	Since $f$ is nonnegative and nonnegative definite, we see that $\phi$ is also nonnegative and nonnegative definite (hence symmetric).
	Moreover, $\Norm{\phi}_{L^2(\mu)} = \Norm{f}_{\cH}=1$.
	By \eqref{E:W=U>0}, we can replace the $U_n(t,f)$ in \eqref{E_:EU} by $W_n(t,\phi)$.
	Then we apply Lemma \ref{L:LaplaceW} to the $dt$ integral to see that
	\begin{align}
		\label{E_:lbdEU}
		\bE  \left[U_n(\tau)\right]\geq \int_{(\bR^d)^n} \prod_{k=1}^{n}\phi(\xi_k) \prod_{k=1}^{n}\frac{1}{1+|\xi_k+\ldots+\xi_n|^2}
		\mu(d\xi_1)\ldots \mu(d\xi_n).
	\end{align}

Following the same arguments as those in Bass, Chen and Rosen \cite{BCR09}, one can prove that (see Appendix C for more details):
	\begin{equation}
		\label{BCR}
		\liminf_{n \to \infty} \frac{1}{n} \log \int_{(\bR^d)^n} \prod_{k=1}^{n}\phi(\xi_k) \prod_{k=1}^{n}\frac{1}{1+|\xi_k+\ldots+\xi_n|^2} \mu(d\xi_1)\ldots \mu(d\xi_n) \geq \log \rho(\phi),
	\end{equation}
where
	\begin{align} \label{E:rhophi}
		\rho(\phi)=\sup_{\Norm{h}_{L^2(\R^d)}=1} \int_{\bR^d} \phi(\xi) \left[\int_{\bR^d} \frac{h(\xi+\eta)h(\eta)}{\sqrt{1+|\xi+\eta|^2}\sqrt{1+|\eta|^2}} d\eta\right] \mu(d\xi).
	\end{align}

  Since both $\phi$ and $\mu$ are nonnegative, the supremum in \eqref{E:rhophi}
  is obtained at some nonnegative $h$. Hence, in the following, we may assume that $h$ is a
nonnegative function. Hence, we see that any $f\in \cH_+$ with $\Norm{f}_{\cH}=1$,
	\begin{align*}
		\liminf_{n \to \infty}\frac{1}{n} \log \E\left[U_n(\tau)\right] \geq \log \rho(\cF f).
	\end{align*}
	We claim that
	\begin{align} \label{E:supHplus}
		\sup_{f\in \mathcal{H}_+;\: \Norm{f}_\mathcal{H}=1} \rho\left(\mathcal{F} f\right) =\sup_{f\in \mathcal{H};\: \Norm{f}_\mathcal{H}=1} \rho\left(\mathcal{F} f\right) = \rho^{1/2},
	\end{align}
	where $\rho$ is defined in \eqref{E:rho}.
	Then an application of \eqref{E:rhoM} proves \eqref{E:EU}.
	\bigskip

	It remains to prove \eqref{E:supHplus}.
	Notice that for any $h\in L^2(\bR^d)$, $\frac{h(\xi)}{\sqrt{1+|\xi|^2}}$ is also in $L^2\left(\bR^d\right)$. Hence the following
	function is well defined:
	\begin{align*}
		g\left(x\right) := g_{h}(x)=\left(2\pi\right)^{d/2} \mathcal{F}^{-1} \left(\frac{h(\cdot)}{\sqrt{1+|\cdot|^2}}\right) \left(x\right).
	\end{align*}
	A first observation is that $g$ is not only in $L^2(\R^d)$ but also in $W^{1,2}(\R^d)\subseteq L^2(\R^d)$ with
	$\Norm{g}_{W^{1,2}(\R^d)}=1$. Indeed,
	\begin{align*}
    \Norm{g}_{W^{1,2}(\R^d)}^2
		= \frac{1}{(2\pi)^{d}} \int_{\R^d} \left(1+|\xi|^2\right) \left|\mathcal{F}g(\xi)\right|^2 d \xi
		= \int_{\R^d} |h\left(\xi\right)|^2 d \xi =1.
	\end{align*}

Note that $\cF (|g|^2)=H * \widetilde{H}$ where $H(\xi)=h(\xi)/\sqrt{1+|\xi|^2}$, since
\begin{align*}
(H * \widetilde{H})(\xi)&=\int_{\bR^d} \frac{h(\xi+\eta)}{\sqrt{1+|\xi+\eta|}}\frac{h(\eta)}{\sqrt{1+|\eta|}}d\eta=\frac{1}{(2\pi)^d}
\int_{\bR^d} \cF g(\xi+\eta) \cF g(\eta)d\eta\\
&=\frac{1}{(2\pi)^d} (\cF g * \widetilde{\cF g})(\xi)=\cF (|g|^2) (\xi),
\end{align*}
using the fact that $\widetilde{\cF g}=\cF \overline{g}$ (because $h$ is real) and the notation $\widetilde{f}(x)=f(-x)$.
If $h\geq 0$, then $\nu(d\xi)=(H * \widetilde{H})(\xi)d\xi$ is a non-negative tempered measure (since $H*\widetilde{H}$ is bounded by $\|h\|_2^2$), and hence, by the Bochner-Schwartz theorem, $|g|^2$ is non-negative-definite. Hence, if $h\geq 0$, then $|g|^2 \in \cH_{+}$.\\
The supremum $\rho\left(\mathcal{F} f\right)$ defined in \eqref{E:rhophi} can be written as:
	\begin{align*}
		\rho\left(\mathcal{F} f\right)
		&= \sup_{\Norm{h}_{L^2(\R^d)}=1,h\geq 0} \int_{\R^d} \cF f (\xi)\cF (|g_h|^2)(\xi) \mu(d\xi) \\
		&= \sup_{\Norm{h}_{L^2(\R^d)}=1,h\geq 0} \int_{\R^d} \left(f*\gamma\right)(x)|g_h|^2(x) d x \\
    &= \sup_{\Norm{h}_{L^2(\R^d)}=1,h\geq 0} \langle f,|g_h|^2 \rangle_{\mathcal{H}}.
	\end{align*}
   Then
  we see that
	\begin{align*}
		\sup_{f\in \mathcal{H};\: \Norm{f}_\mathcal{H}=1} \rho\left(\mathcal{F} f\right)
    & =  \sup_{\Norm{h}_{L^2(\R^d)}=1,h\geq 0} \, \, \sup_{f\in \mathcal{H};\: \Norm{f}_\mathcal{H}=1} \langle f,|g_h|^2 \rangle_{\mathcal{H}}
		 = \sup_{\Norm{h}_{L^2(\R^d)}=1,h\geq 0} \langle f_h^*,|g_h|^2 \rangle_{\mathcal{H}}
	\end{align*}
  where the optimal function is $f_h^*=C|g_h|^2$ with $C=\Norm{|g_h|^2}_\mathcal{H}^{-1}$. Since $f_{h}^* \in \cH_{+}$, we have:
  $$\sup_{f\in \mathcal{H};\: \Norm{f}_\mathcal{H}=1} \rho\left(\mathcal{F} f\right)=
  \sup_{f\in \mathcal{H}_{+};\: \Norm{f}_\mathcal{H}=1} \rho\left(\mathcal{F} f\right)=\sup_{\Norm{h}_{L^2(\R^d)}=1,h\geq 0} \langle f_h^*,|g_h|^2 \rangle_{\mathcal{H}}.$$
  Relation \eqref{E:supHplus} follows noting that
  \begin{align*}
  \langle f_h^*,|g_h|^2 \rangle_{\mathcal{H}}&=\||g_h|^2\|_{\cH}=\left\{\int_{\bR^d} \big|\cF (|g_h|^2 )(\xi)\big|^2 \mu(d\xi)\right\}^{1/2}\\
  &
  =\left\{\int_{\bR^d} \left[ \int_{\bR^d}\frac{h(\xi+\eta)}{\sqrt{1+|\xi+\eta|^2}}\frac{h(\eta)}{\sqrt{1+|\eta|^2}}d\eta
  \right]^2 \mu(d\xi)\right\}^{1/2}.
  \end{align*}

\end{proof}


The next result gives the scaling property of $W_n$ and $U_n$.

\begin{lemma}
	\label{scaling-W}
	Under Assumption \ref{A:Main} with $\alpha<4$, we have that
	$W_n(t)=t^{\frac{4-\alpha}{2}n}W_n(1)$.
	Moreover, when $d\le 3$, $U_n(t)=t^{\frac{4-\alpha}{2}n}U_n(1)$.
	This property also holds under Assumption \ref{A:White} with $\alpha$ replaced by $d$.
\end{lemma}

\begin{proof}
	We consider only the case of Assumption \ref{A:Main}.
	It suffices to prove the case of $W_n(t)$. For any $f\in\cH$, denote $\phi=\cF f\in L^2(\mu)$.
	Using the scaling properties of $\mu$ and $\cF G$ in \eqref{E:scaleMu} and \eqref{E:scaleFG}, respectively,
	we see that for any $c>0$,
	\begin{align*}
		\|\phi(c^{-1}\cdot\,)\|_{L^2(\mu)}^2=c^{\alpha}\|\phi\|_{L^2(\mu)}^2 \quad \mbox{and} \quad
		W_n(t,\phi(c^{-1}\cdot\,))=c^{-(2-\alpha)n}W_n(ct,\phi).
	\end{align*}
	We use these properties for $c=t^{-1}$. Then $\|t^{\alpha/2} \phi(t\,\cdot\,)\|_{L^2(\mu)}^2=\|\phi\|_{L^2(\mu)}^2$ and
	$W_n(t,\phi(t \cdot \,))=t^{n(2-\alpha)}W_n(1,\phi)$.
	We multiply the previous relation by $t^{n\alpha/2}$. Using the fact that $W_n(t,c\phi)=c^n W_n(t,\phi)$ for any $c>0$ and $\phi \in L^2(\mu)$, we obtain:
	$W_n(t,t^{\alpha/2}\phi(t\cdot \,))=t^{\frac{4-\alpha}{2}n}W_n(t,\phi)$.
	Finally, the lemma is proved by taking the supremum over all $f\in \cH$ with $\|f\|_{\cH}=1$.
\end{proof}

\begin{lemma}
	\label{L:asyW1}
	Under Assumption \ref{A:Main} with $0<\alpha<d\leq 3$, it holds that
	\begin{align}
		\label{E:b}
		\liminf_{n \to \infty}\frac{1}{n} \log \left((n!)^{\frac{4-\alpha}{2}} U_n(1) \right)
		\geq \log\left[\left( \frac{2}{4-\alpha}\right)^{\frac{4-\alpha}{2}}2^{-\frac{\alpha}{4}}\cM^{\frac{4-\alpha}{4}} \right].
	\end{align}
	Relation \eqref{E:b} holds also under Assumption \ref{A:White} with $\alpha$ and $\cM$ replaced by $d$ and $\cM(\delta_0)$, respectively.
\end{lemma}

\begin{proof}
  Let $\tau$ denote an exponential random variable of mean one.
	By the scaling property of $U_n$ given by Lemma \ref{scaling-W}, we have:
	\begin{align*}
		\E\left[U_n(\tau)\right]=\bE (\tau^{\frac{4-\alpha}{2}n}) U_n(1)=\Gamma\left(\frac{4-\alpha}{2}n+1\right)U_n(1).
	\end{align*}
	Using Theorem \ref{liminf-EW}, we obtain:
	\begin{align*}
		\liminf_{n \to \infty}\frac{1}{n} \log \left(\Gamma\left(\frac{4-\alpha}{2}n+1\right)U_n(1)\right) \geq \log \left( 2^{-\frac{\alpha}{4}}\cM^{\frac{4-\alpha}{4}}\right).
	\end{align*}
	An application of \eqref{E:StirlingRate} proves the lemma.
\end{proof}

\begin{lemma}
	\label{L:AsyU}
	Under Assumption \ref{A:Main} with $0<\alpha<d\le 3$, for all $a,\theta>0$,
	there exists a constant $c_1=c_1\left(\alpha, \cM, a, \theta\right)>0$
	such that, by setting $n_t=[c_1\: t]$, it holds that
	\begin{align}
		\label{E:tAsyU}
		&\liminf_{t \to \infty} \frac{1}{t} \log \big(a^{n_t}\theta^{n_t/2} U_{n_t}(t)\big)
		\ge \left(a\sqrt{\theta}\right)^{\frac{2}{4-\alpha}} 2^{-\frac{\alpha}{2(4-\alpha)}}\cM^{1/2}.
	\end{align}
	Relation \eqref{E:tAsyU} holds also under Assumption \ref{A:White} with $\alpha$ and $\cM$ replaced by $d$ and $\cM(\delta_0)$, respectively.
\end{lemma}

\begin{proof}
	Let $\varepsilon>0$ be arbitrary. By Lemma \ref{L:asyW1}, there exists $N_{\varepsilon} \in \bN$ such that for all $n\geq N_{\varepsilon}$,
	$(n!)^{\frac{4-\alpha}{2}} U_n(1)\geq \e^{n(\log R'-\varepsilon)}=(R')^n \e^{-n \varepsilon}$,
	where
	\begin{align*}
		R'=\left(\frac{2}{4-\alpha} \right)^{\frac{4-\alpha}{2}} 2^{-\frac{\alpha}{4}} \cM^{\frac{4-\alpha}{4}}.
	\end{align*}
	Fix some $c>0$ and set $n_t:=[ct]$.
	Note that $n_t \geq N_{\varepsilon}$ for any $t \geq t_{\varepsilon}:=(N_{\varepsilon}+1)/c$. For any $t \geq t_{\varepsilon}$,
	\begin{align}
		\label{E_:Asy1}
		a^{n_t}\theta^{n_t/2}U_{n_t}(t)
		=a^{n_t}\theta^{n_t/2}U_{n_t}(1) t^{\frac{4-\alpha}{2}{n_t}}
		\geq (a \sqrt{\theta} R')^{n_t} \frac{1}{(n_t!)^{\frac{4-\alpha}{2}}} t^{\frac{4-\alpha}{2}n_t} \e^{-n_t \varepsilon}.
	\end{align}
	Since $\lim_{t \to \infty}t^{-1}n_t=c$, we deduce that
	\begin{align}
		\label{E_:Asy2}
		\liminf_{t \to \infty}\frac{1}{t}\log \left(a^{n_t}\theta^{n_t/2}U_{n_t}(t)\right)
		= c \liminf_{t \to \infty}\frac{1}{n_t}\log \left(a^{n_t}\theta^{n_t/2}U_{n_t}(t)\right) =:I(n_t).
	\end{align}
	Now by \eqref{E_:Asy1}, we see that
	\begin{align*}
		I\left(n_t\right)
		&\geq c\liminf_{t \to \infty} \frac{1}{n_t} \log \left((a \sqrt{\theta} R')^{n_t} \frac{1}{(n_t!)^{\frac{4-\alpha}{2}}} t^{\frac{4-\alpha}{2}n_t} \right)-c\varepsilon\\
		&=c\log(a\sqrt{\theta}R') +c\liminf_{t \to \infty} \frac{1}{n_t}
			\log \left( \left(\frac{t}{n_t}\right)^{\frac{4-\alpha}{2}n_t}
				\frac{n_t^{\frac{4-\alpha}{2}n_t}}{(n_t!)^{\frac{4-\alpha}{2}}}\right)-c\varepsilon\\
		&=c\log\left(a\sqrt{\theta}R'\right)-c\frac{4-\alpha}{2}\log c+c\frac{4-\alpha}{2}\liminf_{t\to \infty}\frac{1}{n_t}\log \frac{n_t^{n_t}}{n_t!}-c\varepsilon\\
		&=c\log\left(a\sqrt{\theta}R'\right)-c\frac{4-\alpha}{2}\log c+c\frac{4-\alpha}{2}-c\varepsilon,
	\end{align*}
	where the last equality is due to Stirling's formula. Let $\varepsilon \to 0$. Relation \eqref{E:tAsyU} follows by maximizing over $c$. The optimal $c$ is $(a\sqrt{\theta}R')^{2/(4-\alpha)}$.
\end{proof}

\bigskip
We are now ready to give the proof of the desired lower bound:

\begin{proof}[Proof of Theorem~\ref{T:Lowbound}.]
	We consider only the case of Assumption \ref{A:Main}. The white noise case can be proved in a similar way.
	We return to Proposition~\ref{P:LB}.
	In relation~\eqref{LB-2}, we take the supremum over all $f\in\cH_+$ with $\|f\|_{\cH}=a>0$.
	Using the fact that $W_n(t,\phi)=a^n W_n(t,\phi/a)$ for any $\phi \in L^2(\mu)$ and the nonnegativity in
	\eqref{E:W=U>0}, for any $c>0$, we have that
	\begin{align*}
		\Norm{u(t,0)}_p
		& \geq \exp \left\{-\frac{1}{2} t_p^{\beta}a^2 \right\}
			\sup_{f\in \cH_+;\: \Norm{f}_\cH=a} \sum_{n\geq 0} \theta^{n/2} W_n\left(t_p^{\beta}, \cF f\right)\\
		&=\exp \left\{-\frac{1}{2} t_p^{\beta}a^2 \right\}
			\sup_{f\in \cH_+;\: \Norm{f}_\cH=1} \sum_{n\geq 0} a^n \theta^{n/2} W_n\left(t_p^{\beta}, \cF f\right)\\
		&\ge\exp \left\{-\frac{1}{2}  t_p^{\beta}a^2 \right\}
			\sup_{f\in \cH_+;\: \Norm{f}_\cH=1} a^{n_t} \theta^{n_t/2} U_{n_t}\left(t_p^{\beta},f\right)\\
&=\exp \left\{-\frac{1}{2}  t_p^{\beta}a^2 \right\} a^{n_t} \theta^{n_t/2}U_{n_t}(t_{p}^{\beta}),
	\end{align*}
	where $n_t=[c\: t_p^\beta]$.
	By choosing $c$ to be the constant in Lemma~\ref{L:AsyU}, we see that
	\begin{align*}
		\liminf_{t_p\to \infty}\frac{1}{t_p^{\beta}}\log \|u(t,0)\|_p
		\geq - \frac{1}{2}a^2 +(a\sqrt{\theta})^{\frac{2}{4-\alpha}} 2^{-\frac{\alpha}{2(4-\alpha)}}\cM^{1/2}=:h(a).
	\end{align*}
	We now maximize over $a>0$.
	More precisely, let $b=\theta^{\frac{1}{4-\alpha}}2^{-\frac{\alpha}{2(4-\alpha)}}\cM^{1/2}$.
	The maximum of the function $h(a)=-\frac{1}{2} a^2+b a^{\frac{2}{4-\alpha}},a>0$ is attained
	at the point $a^*=\left(\frac{2b}{4-\alpha}\right)^{\frac{4-\alpha}{2(3-\alpha)}}$ and the maximum value of $h$ is:
	\begin{align*}
		h(a^*)=(3-\alpha)(4-\alpha)^{-\frac{4-\alpha}{3-\alpha}}
		2^{\frac{1}{3-\alpha}}
		b^{\frac{4-\alpha}{3-\alpha}}.
	\end{align*}
	Plugging the value $b$ proves the theorem.
\end{proof}

\appendix
\section{Exponential behaviour of power series}

In this part we examine the asymptotic behaviour of some power series, measured on the exponential scale.
The first result is a useful tool for comparing two power series.
Its proof follows using the same arguments as those in the proof of Theorem 1.3 of Balan and Song \cite{balan-song19}.

\begin{lemma}
	\label{a-b-lemma}
	If $(a_n)_{n \geq 0}$ and $(b_n)_{n\geq 0}$ are two sequences of positive real numbers such that
	\begin{align*}
		\liminf_{n \to \infty} \frac{1}{n} \log \frac{b_n}{a_n} \geq \log \rho \quad \mbox{and} \quad
		\liminf_{t \to \infty}\frac{1}{t^p} \log \sum_{n\geq 0} a_n t^n \geq A
	\end{align*}
	for some $\rho>0,p>0$ and $A \in \bR$, then
	$\liminf_{t \to \infty}\frac{1}{t^p} \log \sum_{n\geq 0} b_n t^n \geq \rho^p A$.
	The same statement remains valid if we replace $(\liminf,\geq)$ by $(\limsup,\leq)$ or $(\lim,=)$.
\end{lemma}

The next lemma is an application of the previous result.
It says that, if $R_n^{1/n} \approx R$ as $n \to \infty$,
then we can replace $R_n$ by $R^n$ when deriving the exponential asymptotic behaviour of the power series $\sum_{n\geq 0} x_n R_n t^n$.

\begin{lemma}
	\label{Rn-compare}
	If $(R_n)_{n\geq 0}$ and $(x_n)_{n\geq 0}$ are two sequences of positive real numbers such that
	\begin{align*}
		\liminf_{n \to \infty} \frac{1}{n} \log R_n \geq \log R \quad \mbox{and} \quad
		\liminf_{t \to \infty}\frac{1}{t^p} \log \sum_{n\geq 0} x_n R^n t^n \geq A
	\end{align*}
	for some $R>0,p>0$ and $A \in \bR$, then
	$\liminf_{t \to \infty} t^{-p} \log \sum_{n\geq 0} x_n R_n t^n \geq A$.
	The same statement remains valid if we replace $(\liminf,\geq)$ by
	$(\limsup,\leq)$ or $(\lim,=)$.
\end{lemma}

\begin{proof}
	The lemma is proved by an application of Lemma \ref{a-b-lemma} with $a_n=x_n R^n$ and $b_n=x_n R_n$,
	and the fact that
	$\liminf_{n \to \infty} \frac{1}{n} \log \frac{R_n}{R^n}=\liminf_{n \to \infty}\frac{1}{n}\log R_n-\log R \geq 0$.
\end{proof}

The next lemma is about the asymptotic property of the {\it Mittag-Leffler} function \cite{Pod99Frac}.

\begin{lemma}
	\label{L:Gamma}
	For any $\gamma >0$,
	$\lim_{t \to \infty}t^{-1/\gamma} \log  \sum_{n\geq 0} (n!)^{-\gamma} t^n=\gamma$.
\end{lemma}

\begin{proof}
	We start by looking at the similar result with $(n!)^{\gamma}$ replaced by $\Gamma(\gamma n +1)$:
	\begin{align*}
		\lim_{t \to \infty}\frac{1}{t^{1/\gamma}}\sum_{n\geq 0} \frac{t^n}{\Gamma(\gamma n+1)}
		=\lim_{t \to \infty}\frac{1}{t} \sum_{n\geq 0} \frac{t^{\gamma n}}{\Gamma(\gamma n+1)}
		=1,
	\end{align*}
	where the last equality is due to the asymptotic property of the Mittag-Leffler function;
	see Theorem 1.3 on p. 32 and Theorem 1.7 on p. 35 of \cite{Pod99Frac}.
	This lemma is proved by an application of \eqref{E:StirlingRate} and Lemma \ref{a-b-lemma}
	with $a_n=\frac{1}{\Gamma(\gamma n+1)}$, $b_n=\frac{1}{(n!)^{\gamma}}$ and $p=1/\gamma$.
\end{proof}

\section{Moment comparison using hypercontractivity}

This section gives the moment comparison result of Khoa L\^e \cite{le16}.
This result was stated in \cite{le16} for the Parabolic Anderson Model with a special Gaussian noise,
but L\^e's proof is in fact valid in a much more general case. We present this proof here,
including some of the details which are missing from \cite{le16}.

Let $W=\{W(\varphi);\varphi \in \cH\}$ be an isonormal Gaussian process,
associated to a Hilbert space $\cH$. Assume that either one of the following conditions hold:\\
(i) $\cH$ consists  of functions (or distributions) on $\bR_{+} \times \bR^d$,
i.e. $W$ is {\em time-dependent}; and\\
(ii)$\cH$ consists of functions (or distributions) on $\bR^d$, i.e. $W$ is {\em time-independent}.\\

Let $\cal L$ be a second-order pseudo-differential operator of $\bR_{+} \times \bR^d$ and $u_{\theta}$ be the solution of the SPDE:
\begin{equation}
	\label{E:Leq}
	{\cal L}u(t,x)=\sqrt{\theta}u(t,x) \dot{W}, \quad t>0,x\in \bR^d
\end{equation}
with (deterministic) initial condition.
By definition, the {\em (mild Skorohod) solution} to \eqref{E:Leq} is an adapted square-integrable process $u_{\theta}=\{u_{\theta}(t,x);t>0,x\in \bR^d\}$ which satisfies
\begin{align*}
	u_{\theta}(t,x)=J_0(t,x)+\sqrt{\theta}\int_0^t \int_{\bR^d}G(t-s,x-y)u_{\theta}(s,y)W(\delta s,\delta y),
\end{align*}
if the noise $W$ is time-dependent, respectively
\begin{align*}
	u_{\theta}(t,x)=J_0(t,x)+\sqrt{\theta}\int_0^t \int_{\bR^d}G(t-s,x-y)u_{\theta}(s,y)W(\delta y)ds,
\end{align*}
if the noise $W$ is time-independent, provided that these integrals are well-defined.
Here $W(\delta s,\delta y)$ (respectively $W(\delta y)$) denotes the Skorohod integral with respect to $W$,
$G$ is the fundamental solution of $\cal L$ on $\bR_{+} \times \bR^d$,
and $J_0$ is the solution of the deterministic equation ${\cal L}u=0$ on
$\bR_{+} \times \bR^d$, with the same initial condition as \eqref{E:Leq}.

\begin{theorem}
	\label{T:Khoa}
	If for any $\theta>0$, equation \eqref{E:Leq} has a unique solution $u_{\theta}=\{u_{\theta}(t,x);t>0,x\in \bR^d\}$
	and $\bE\left(|u_{\theta}(t,x)|^p\right)<\infty$ for any $t>0$, $x \in \bR^d$ and $p>1$, then
	\begin{equation}
		\label{mom-comp}
		\Norm{u_{\frac{p-1}{q-1}\theta}(t,x)}_{q}\leq \|u_{\theta}(t,x)\|_{p} \quad \mbox{for any} \ 1<p \leq q.
	\end{equation}
	In particular,
	\begin{align*}
		\|u_{\theta}(t,x)\|_{p}\leq \|u_{(p-1)\theta}(t,x)\|_2 \quad \mbox{for any} \ p\geq 2.
	\end{align*}
\end{theorem}

\begin{proof}
  The proof is based on \textit{Mehler's formula};
  see for instance, (1.67) of \cite{nualart06} or (4.11) of \cite{Janson97}.
	First, we need to introduce the framework for this result, as presented in Section 1.4 of Nualart \cite{nualart06}.
	Suppose that $W$ is defined on a complete probability space $(\Omega,\cF,\bP)$, where $\cF$ is generated by $W$.
	Let $W'=\{W'(\varphi);\varphi \in \cH\}$ be a copy of $W$, defined on $(\Omega',\cF',\bP')$,
	where $\cF'$ is generated by $W'$.
	Using the projection maps, we redefine $W$ and $W'$ on $(\Omega \times \Omega', \cF \times \cF',\bP \times \bP')$,
	so that $W$ and $W'$ are independent.

	For any $\tau>0$ and $\varphi \in \cH$, let
	\begin{align*}
		Z_{\tau}(\varphi)= \e^{-\tau}W(\varphi)+\sqrt{1-\e^{-2\tau}}W'(\varphi).
	\end{align*}
	Then $Z_{\tau}=\{Z_{\tau}(\varphi);\varphi \in \cH\}$ is also an isonormal Gaussian process, defined on $\Omega \times \Omega'$.

	For any $\cF$-measurable function $F:\Omega \to \bR$, there exists a measurable map $\psi_F:\bR^{\cH} \to \bR$ such that $F=\psi_F(W)$.
  Let $(T_{\tau})_{\tau \geq 0}$ be the Ornstein-Uhlenbeck semigroup on $L^2(\Omega)$.
  By Mehler's formula, for any $\tau>0$ and $F \in L^2(\Omega)$
	\begin{align*}
		T_{\tau}(F)=\bE'\left[\psi_{F}\left(Z_{\tau}\right)\right] \quad \mbox{$\bP$-a.s.},
	\end{align*}
	where $\bE'$ denotes the expectation with respect to $\bP'$.
	We apply this formula to $F=u_{\theta}(t,x)$. We denote $\psi_{u_{\theta}(t,x)}=\psi_{\theta,t,x}$.
	Then $u_{\theta}(t,x)=\psi_{\theta,t,x}(W)$ and
	\begin{equation}
		\label{Mehler}
		T_{\tau}(u_{\theta}(t,x))=\bE'\left[\psi_{\theta,t,x}\left(Z_{\tau}\right)\right]  \quad \mbox{$\bP$-a.s.}
	\end{equation}

	Since $\left\{\psi_{\theta,t,x}(W);t\geq 0,x \in \bR^d\right\}$ is a solution to \eqref{E:Leq},
	$\left\{\psi_{\theta,t,x}(Z_{\tau});t\geq 0,x \in \bR^d\right\}$ is a solution of
	\begin{equation}
		\label{E:Leq2}
		{\cal L}u(t,x)=\sqrt{\theta} u(t,x) \dot{Z}_{\tau}, \quad t>0,\: x\in \bR^d,
	\end{equation}
	with the same initial condition as \eqref{E:Leq}.
	Denote $u_{\tau,\theta}(t,x):=\psi_{\theta,t,x}(Z_{\tau})$. Then  \eqref{Mehler} becomes:
	\begin{equation}
		\label{Mehler2}
		T_{\tau}(u_{\theta}(t,x))=\bE'[u_{\tau,\theta}(t,x)]  \quad \mbox{$\bP$-a.s.}
	\end{equation}

	Assume that the noise $W$ is time-dependent. (The case of the time-independent noise is similar.)
	Since $\left\{u_{\tau,\theta}(t,x);t\geq 0,x\in \bR^d\right\}$ is a solution of \eqref{E:Leq2} with the same initial
	condition as \eqref{E:Leq}, we see that
	\begin{align*}
		u_{\tau,\theta}(t,x)
		& =J_0(t,x)+\sqrt{\theta} \int_0^t \int_{\bR^d} G(t-s,x-y) u_{\tau,\theta}(s,y) Z_{\tau}(\delta s,\delta y) \\
		& =J_0(t,x)+\sqrt{\e^{-2\tau} \theta}\int_0^t \int_{\bR^d} G(t-s,x-y) u_{\tau,\theta}(s,y) W(\delta s,\delta y)\\
		& \quad+\sqrt{\theta(1-\e^{-2\tau})} \int_0^t \int_{\bR^d} G(t-s,x-y) u_{\tau,\theta}(s,y) W'(\delta s,\delta y).
	\end{align*}
	We take expectation with respect to $\bP'$.
	The third term on the right-hand side above disappears since the Skorohod integral has zero mean.
	Since $W$ and $W'$ are independent, the expectation with respect to $\bP'$ commutes with the Skorohod integral with respect to
	$W$ and hence
	\begin{align*}
		\bE'[u_{\tau,\theta}(t,x)]
		=J_0(t,x)+\sqrt{\e^{-2\tau} \theta}\int_0^t \int_{\bR^d} G(t-s,x-y) \bE'[u_{\tau,\theta}(s,y)] W(\delta s,\delta y).
	\end{align*}
  This proves that the process $\left\{ \bE'\left[u_{\tau,\theta}\left(t,x\right)\right];\: t\geq 0,x\in \bR^d\: \right\}$ is a solution of
	\begin{align*}
		\cL u=\sqrt{\e^{-2\tau}\theta}  u \dot{W}, \quad t>0,x\in \bR^d
	\end{align*}
	with the same initial condition as \eqref{E:Leq}.
	Since this equation has the {\em unique} solution $u_{\e^{-2\tau} \theta}$, we conclude that for any $t>0$ and $x \in \bR^d$,
	\begin{align*}
		\bE'[u_{\tau,\theta}(t,x)]=u_{\e^{-2\tau} \theta}(t,x) \quad \mbox{$\bP$-a.s.}
	\end{align*}
	Combining this with \eqref{Mehler2}, we obtain that for any $\tau>0$, $t>0$ and $x \in \bR^d$,
	\begin{equation}
		\label{Mehler3}
		T_{\tau}(u_{\theta}(t,x))=u_{\e^{-2\tau} \theta}(t,x) \quad \mbox{$\bP$-a.s.}
	\end{equation}

	By the hypercontractivity of the Ornstein-Uhlenbeck semigroup
  (see, e.g., Theorem 1.4.1 of Nualart \cite{Janson97} or Theorem 5.1 of Janson \cite{nualart06}),
	for any $\tau>0$, $p>1$ and $F \in L^p(\Omega)$,
	\begin{align*}
		\|T_{\tau}F\|_{q(\tau)} \leq \|F\|_p
	\end{align*}
	where $q(\tau)= \e^{2\tau}(p-1)+1$.
	We apply this to $F=u_{\theta}(t,x)$.
	Using \eqref{Mehler3}, we obtain:
	\begin{align*}
		\Norm{u_{\e^{-2\tau} \theta}(t,x)}_{q(\tau)} \leq \|u_{\theta}(t,x)\|_p.
	\end{align*}
	For fixed $1<p \leq q$, choose $\tau>0$ such that $q(\tau)=q$.
	Then $e^{-2\tau}=\frac{p-1}{q-1}$ and \eqref{mom-comp} follows.
\end{proof}

\section{Proof of \eqref{BCR}}

Under Assumption A, the density function $\varphi$ of $\mu$ satisfies the scaling property \eqref{E:scaleMu}.  Using a change of variables, and letting $\eta_0=0$, we have
\begin{align}
  \nonumber       & \int_{(\bR^d)^n} \prod_{j=1}^{n} \frac{1}{1+|\xi_j+\ldots+\xi_n|^{2}} \phi(\xi_j) \varphi(\xi_j) d\xi_1 \ldots d\xi_n= \\
  \label{eta-int} & \quad \int_{(\bR^d)^n} \prod_{j=1}^{n} \frac{1}{1+|\eta_j|^{2}} \phi(\eta_j-\eta_{j-1}) \varphi(\eta_j-\eta_{j-1}) d\eta_1 \ldots d\eta_n.
\end{align}
The last integral coincides with the integral on the far right-hand-side of equation (3.3) of \cite{BCR09}, in which $f$ is replaced by $\phi$. We argue as on pages 636-637 of \cite{BCR09} and define the linear operator $T:L^2(\bR^d) \to L^2(\bR^d)$:
\[
(Tg)(\eta)=\frac{1}{\sqrt{1+|\eta|^{2}}}\int_{\bR^d} \frac{1}{\sqrt{1+|\xi|^{2}}} \phi(\xi-\eta) \varphi(\xi-\eta) g(\xi)d\xi, \quad g \in L^2(\bR^d).
\]
This operator is self-adjoint, and has the spectral representation:
\[
\langle Tg,g \rangle=\int_{\bR}\theta \mu_{g}(d\theta) \quad \mbox{for all} \quad g \in L^2(\bR^d),
\]
where $\mu_g$ is a probability measure on $\bR$.
 Since $\varphi$ and $\psi$ are non-negative definite, so is their product. Therefore, the operator $T$ is non-negative-definite. This implies that $\mu_g$ has support in $(0,\infty)$ for any $g$.
For any $g,h \in L^2(\bR^d)$,
\[
\langle h,Tg \rangle=\int_{\bR^d} \left(\int_{\bR^d} \frac{h(\eta)}{\sqrt{1+|\eta|^{2}}}
\frac{g(\xi+\eta)}{\sqrt{1+|\xi+\eta|^{2}}}d\eta\right) \phi(\xi)\varphi(\xi)d\xi.
\]
Assume that $g\in L^2(\bR^d)$ is bounded, has compact support $K$ and $\|g\|_{L^2(\bR^d)}= 1$. There is a $\delta>0$ such that $\phi, \varphi,Q \geq \delta $ on $K$, where $Q(\xi)=1/\sqrt{1+|\xi|^{2}}$.
By (3.7) and (3.9) of \cite{BCR09},
\[
\int_{(\bR^d)^n} \prod_{j=1}^{n} \frac{1}{1+|\eta_j|^{2}} \phi(\eta_j-\eta_{j-1}) \varphi(\eta_j-\eta_{j-1}) d\eta_1 \ldots d\eta_n \geq \delta^3 \|g\|_{L^{\infty}(\mathbb{R}^d)}^{-2} \left\langle g,Tg \right\rangle^{n-1}.
\]
Note that relation (3.9) of \cite{BCR09} says that $\int_{-\infty}^{\infty}\theta^{n-1}\mu_{g}(d\theta) \geq \left(\int_{-\infty}^{\infty} \theta \mu_g(d\theta)\right)^{n-1}$
(due to Jensen's inequality). For this inequality, we need that $\mu_{g}$ has support in $(0,\infty)$.

Using \eqref{eta-int}, it follows that
\[
\liminf_{n\to \infty}\frac{1}{n}\log \int_{(\bR^d)^n} \prod_{j=1}^{n} \frac{1}{1+|\xi_j+\ldots+\xi_n|^{2}} \phi(\xi_j) \varphi(\xi_j) d\xi_1 \ldots d\xi_n \geq \log \left\langle g,Tg \right\rangle.
\]
Since the set of bounded, compactly supported functions in $L^2(\bR^d)$ is dense in $L^2(\bR^d)$, the previous inequality holds for any $g \in L^2(\bR^d)$ with $\|g\|_{L^2(\bR^d)}=1$. 
We take the supremum over all functions $g \in L^2(\bR^d)$ with $\|g\|_{L^2(\bR^d)}=1$. The conclusion follows.

\bigskip

{\bf Acknowledgement.} We would like to thank two anonymous referees who read the paper very carefully and made numerous suggestions for improvement.

\end{document}